\documentclass[11pt,reqno,a4paper]{amsart}
\usepackage{a4wide,verbatim}

\title[Minimal surfaces in $\RH^4$]{The moduli spaces of equivariant minimal surfaces in $\RH^3$ and $\RH^4$ via Higgs bundles}

\author{John Loftin}
\address{Department of Mathematics and Computer Science\\ Rutgers-Newark\\
Newark, NJ 07102, USA}
\email{loftin@rutgers.edu}

\author{Ian McIntosh}
\address{Department of Mathematics\\ University of York\\
York YO10 5DD, UK}
\email{ian.mcintosh@york.ac.uk}
\date{April 23, 2018}

\newcommand{\Z}{\mathbb{Z}}

\newcommand{\C}{\mathbb{C}}
\newcommand{\Ct}{\mathbb{C}^\times}

\newcommand{\R}{\mathbb{R}}

\newcommand{\CP}{\mathbb{CP}}
\newcommand{\CH}{\mathbb{CH}}
\newcommand{\RH}{\mathbb{RH}}

\renewcommand{\P}{\mathbb{P}}

\newcommand{\caD}{\mathcal{D}}
\newcommand{\caE}{\mathcal{E}}
\newcommand{\caF}{\mathcal{F}}

\newcommand{\caH}{\mathcal{H}}

\newcommand{\caM}{\mathcal{M}}

\newcommand{\caR}{\mathcal{R}}
\newcommand{\caS}{\mathcal{S}}
\newcommand{\caT}{\mathcal{T}}

\newcommand{\caV}{\mathcal{V}}
\newcommand{\caW}{\mathcal{W}}

\newcommand{\fE}{\mathfrak{E}}

\newcommand{\Aut}{\operatorname{Aut}}

\newcommand{\End}{\operatorname{End}}

\newcommand{\Hom}{\operatorname{Hom}}
\newcommand{\Sec}{\operatorname{Sec}}

\newcommand{\Inn}{\operatorname{Inn}}
\newcommand{\Out}{\operatorname{Out}}

\newcommand{\Pic}{\operatorname{Pic}}
\newcommand{\Isom}{\operatorname{Isom}}

\newcommand{\Spin}{\operatorname{Spin}}

\newcommand{\im}{\operatorname{im}}

\newcommand{\ad}{\operatorname{ad}}

\newcommand{\tr}{\operatorname{tr}}

\newcommand{\h}[2]{\langle{#1},{#2}\rangle}
\newcommand{\II}{\mathrm{I\! I}}
\newcommand{\bz}{{\bar z}}
\newcommand{\bpartial}{{\bar\partial}}

\newcommand{\0}{\mathbf{0}}
\newcommand{\ocaM}{\overline{\mathcal{M}}}

\newtheorem{thm}{Theorem}[section]
\newtheorem{prop}[thm]{Proposition}
\newtheorem{lem}[thm]{Lemma}
\newtheorem{cor}[thm]{Corollary}

\newtheorem{defn}[thm]{Definition}

\theoremstyle{remark}

\newtheorem{rem}{Remark}[section]

\numberwithin{equation}{section}

\begin{document}

\begin{abstract}
In this article we introduce a definition for the moduli space of equivariant minimal
immersions of the Poincar\'e disc into a non-compact symmetric space, where
the equivariance is with respect to representations of the fundamental group of a
compact Riemann surface of genus at least two. We then study
this moduli space for the non-compact symmetric space $\RH^n$ and show how $SO_0(n,1)$-Higgs bundles
can be used to parametrise this space, making clear how the classical
invariants (induced metric and second fundamental form) figure in this picture. We use
this parametrisation to provide details of the moduli spaces for
$\RH^3$ and $\RH^4$, and relate their structure to the structure of the corresponding Higgs
bundle moduli spaces.
\end{abstract}

\maketitle

\section{Introduction}

In this article we study equivariant minimal immersions $f:\caD\to N$ of the Poincar\'{e} disc $\caD$ into a non-compact
symmetric space $N$. By ``equivariant'' we mean each immersion intertwines two
representations, into the isometry groups of $\caD$ and $N$ respectively, of the fundamental group of a closed oriented 
surface $\Sigma$ of genus at least two. Since there can be no compact minimal surfaces in a non-compact (globally) symmetric
space the equivariant minimal discs are the closest analogy to compact minimal surfaces in $N$.

Our aim here is to introduce a definition for the moduli space of equivariant minimal discs and then
show how the \emph{non-abelian Hodge correspondence} can be used to parametrise this moduli space when $N=\RH^n$. 
If one takes the classical point of view it is very hard to describe the moduli space of such minimal
immersions: the classical data is the induced metric, a
metric connexion in the normal bundle, and the second fundamental form, all three linked by the Gauss-Codazzi-Ricci equations, which are
a system of non-linear p.d.e. Even for minimal surfaces
in $\RH^3$ the description using this approach is challenging (cf. \cite{Tau}). By contrast, we will show that by exploiting the 
non-abelian Hodge correspondence we can use Higgs bundles to provide a
parametrisation of these minimal surfaces using purely holomorphic data. The most difficult part of the parametrisation is
to understand the conditions on the parameters which make the Higgs bundles (poly)stable: we give the details of this for
$n=3,4$. The ideas here put into broader context, and are the 
natural extension of, our study of minimal surfaces in $\CH^2$ \cite{LofM15}.

To explain our approach, first recall that the non-abelian Hodge correspondence provides a
homeomorphism between two different moduli spaces. On one side is the \emph{character variety} $\caR(\pi_1\Sigma,G)$
of a non-compact semi-simple Lie group $G$, where $\Sigma$ is a smooth closed oriented surface of genus $g\geq 2$.
The character variety is the moduli space of $G$-conjugacy classes of reductive representations of $\pi_1\Sigma$
in $G$. On the other side is the moduli space $\caH(\Sigma_c,G)$ of
\emph{polystable $G$-Higgs bundles} over a compact Riemann surface $\Sigma_c$. Here we
use $\Sigma_c$ to denote $\Sigma$ equipped with a complex structure.
The space $\caH(\Sigma_c,G)$ parametrises solutions of an appropriate version of the self-dual Yang-Mills
equations over $\Sigma_c$.
These moduli spaces are homeomorphic (and diffeomorphic away from singularities), but $\caH(\Sigma_c,G)$ also has a
complex structure which depends upon $\Sigma_c$. This correspondence developed from the seminal
work of Hitchin \cite{Hit87}, Donaldson \cite{Don}, Corlette \cite{Cor} and Simpson \cite{Sim88} (with
the case of stable Higgs bundles for arbitrary real reductive groups proven in \cite{BGPM}).  The half of the non-abelian Hodge correspondence due to Hitchin and Simpson, in which the polystable Higgs bundle is shown to produce an equivariant harmonic map, is the
Higgs bundle case of the Donaldson-Uhlenbeck-Yau correspondence.
There are now many good surveys available of the principal results of non-abelian Hodge theory (see, for example,
\cite{Garapp,Gar,Wen}).

The link to minimal surfaces is that the non-abelian Hodge correspondence already tells us about all
equivariant (or ``twisted'') harmonic maps: fix a Fuchsian
representation $c$ of $\pi_1\Sigma$ into the group $\Isom^+(\caD)$ of oriented isometries of the Poincar\'e disc $\caD$,
so that $\Sigma_c\simeq \caD/c(\pi_1\Sigma)$. For every irreducible
representation $\rho:\pi_1\Sigma\to G$, there is a unique equivariant harmonic map $f:\caD\to N$ 
into the non-compact symmetric space $N$ associated to $G$ \cite{Don,Cor} (equivariance means that
$f\circ c(\delta) = \rho(\delta)\circ f$ for all $\delta\in\pi_1\Sigma$). A
map from a surface is minimal when it is both conformal and harmonic, and it is easy to show that
$f$ is conformal when the Higgs field $\Phi$ of the corresponding Higgs bundle satisfies $\tr(\ad\Phi^2)=0$ and
$\Phi$ does not vanish (more generally, $f$ will have branch points at zeroes of $\Phi$). Note that when $\rho$
is discrete we obtain an incompressible minimal immersion of $\Sigma$ into the locally symmetric space
$N/\rho(\pi_1\Sigma)$.

To obtain all equivariant minimal immersions one must allow $c$ to range over all Fuchsian representations. The
natural equivalence for equivariant immersions means we only care about the conjugacy class of $c$, and these are
parametrised by the Teichm\"{u}ller space $\caT_g$ of $\Sigma$. We will also insist that the representation $\rho$
is irreducible, for otherwise the same map $f$ can be equivariant with respect to more than one representation $\rho$.
For the case where $N=\RH^n$
(or $\CH^n$) this is equivalent to the condition that $f$ is \emph{linearly full}, i.e., does not take values in a
lower dimensional totally geodesic subspace.
With this assumption we can embed the moduli space $\caM(\Sigma,N)$ of such equivariant minimal immersions into the
product space $\caT_g\times \caR(\pi_1\Sigma,G)$:  the details of this are explained in \S 2.1 below. 

From \S 2.2 onwards we focus on the case where $N=\RH^n$, $G=SO_0(n,1)$ is the group of orientation preserving isometries, and $f$ is
oriented (i.e., we fix an isomorphism $\wedge^{n-2}T\Sigma^\perp\simeq 1$). 
We begin by showing how the induced metric and second fundamental form of an equivariant minimal surface in $\RH^n$ determines the
Higgs bundle for $\rho$. From the work of Aparacio \& Garc\'ia-Prada \cite{ApaG}, we know that each $SO_0(n,1)$-Higgs bundle 
$(E,\Phi)$ has $E=V\oplus
1$ where $V$ is an $SO(n,\C)$-bundle (i.e., rank $n$ with trivial determinant line bundle and an orthogonal structure $Q_V$) and
$1$ denotes the trivial bundle. We show
in Theorem \ref{thm:Higgs} that this is the Higgs bundle for an equivariant minimal immersion if and only if it is polystable
and $V$ is constructed from an $SO(n-2,\C)$ bundle $(W,Q_W)$ and a cohomology class $\xi\in H^1(\Hom(W,K^{-1}))$, where $K$ is the
canonical bundle of $\Sigma_c$, as follows. As a smooth bundle $V$ is just
$K^{-1}\oplus W\oplus K$. This has a natural orthogonal structure $Q_V$ given by $Q_W$ on $W$
together with the canonical pairing $K\times K^{-1}\to 1$ on $K^{-1}\oplus K$. The class $\xi$
is the extension class of a rank $n-1$ holomorphic bundle $V_{n-1}$
\[
0\to K^{-1}\stackrel{\phi}{\to} V_{n-1}\to W\to 0.
\]
We show there is a unique holomorphic structure $V$ on $K^{-1}\oplus W\oplus K$ for which: (a) the projection
$V\to K$ is holomorphic and has kernel isomorphic to $V_{n-1}$, (b) the orthogonal structure
$Q_V$ is holomorphic. The Higgs field $\Phi$ is then determined by $\phi$ and its dual with respect to $Q_V$. 
However, the group $SO(Q_W)$ of isomorphisms of $(W,Q_W)$ leaves invariant the isomorphism class of the Higgs bundle,
so that two extension classes $\xi$ and $\xi'$ determine the same
equivariant minimal immersion when $\xi'=g\cdot \xi$ for $g\in SO(Q_W)$ (the action is by pre-composition on sections of
$\Hom(W,K^{-1})$). We will denote the corresponding equivalence class by 
\[
[\xi]\in H^1(\Hom(W,K^{-1}))/SO(Q_W).
\]
We show that, in terms of the geometry of the minimal surface, $W$ is just the complexified normal bundle of the immersion, 
and $\xi$ comes from the $(0,2)$ part of the second fundamental form. 

To use $(W,Q_W,[\xi])$ to parametrise $\caM(\Sigma,\RH^n)$ we must know when the Higgs
bundle it produces is polystable (and indecomposable, to ensure that the representation is irreducible). At present we can
only do this in
detail when $n=3,4$. Indeed, the case $n=3$ is very easy since $W$ is trivial. As a consequence we quickly recover a description 
of $\caM(\Sigma,\RH^3)$ which was originally due to Alessandrini \& Li \cite{AL}: we include it here since this moduli
space plays a role in the boundary of $\caM(\Sigma,\RH^4)$. 
The parametrisation shows that $\caM(\Sigma,\RH^3)$ it is diffeomorphic to the punctured tangent bundle of Teichm\"uller space, i.e, 
$T\caT_g$ with its zero section removed.
Moreover, the zero section corresponds to the totally geodesic minimal immersions, which necessarily lie in a copy of $\RH^2$
and are not linearly full. 
The moduli space of minimal surfaces in $3$-dimensional hyperbolic space
forms was also studied by Taubes \cite{Tau}, using the more traditional approach of classifying these
surfaces by their metric and Hopf differential via the Gauss-Codazzi equations.
We explain, in Remark \ref{rem:n=3}, how our results fit in with Taubes' space of
``minimal hyperbolic germs''.

The structure of $\caM(\Sigma,\RH^4)$ is more interesting. In this case the normal bundle $T\Sigma^\perp$ has an Euler number
$\chi(T\Sigma^\perp)$ and we show that this integer invariant is bounded and indexes the connected components of the moduli
space. To be precise, we prove:
\begin{thm}
The moduli space $\caM(\Sigma,\RH^4)$ can be given the structure of a non-singular complex manifold of dimension
$10(g-1)$. It
has $4g-5$ connected components $\caM_l(\Sigma,\RH^4)$, with integer index satisfying $|l|<2(g-1)$. The component
$\caM_l(\Sigma,\RH^4)$ consists of all linearly full minimal immersions whose normal bundle has $\chi(T\Sigma^\perp)=l$.
\end{thm}
It is worth remarking that the character variety $\caR(\pi_1\Sigma,SO_0(4,1))$ has the same real
dimension as $\caM(\Sigma,\RH^4)$. This begs a general question, which cannot be addressed here, as to the properties of
the projection from $\caM(\Sigma,N)$ to $\caR(\pi_1\Sigma,G)$ and particularly whether the image is open.

Our understanding of the structure goes beyond counting connected components. First, we show that each component is an open
subvariety of a complex analytic family over Teichm\"uller space. We denote
the fibre of $\caM(\Sigma,\RH^4)$ over $c\in\caT_g$ by $\caM(\Sigma_c,\RH^4)$. It is isomorphic to an open subvariety of
the nilpotent cone in $\caH(\Sigma_c,G)$, namely, those indecomposable Higgs bundles with $\tr(\ad\Phi^2)=0$. 
By Hausel's Theorem \cite{Hau} the nilpotent cone is a union
of the unstable manifolds of the downward gradient flow for the Higgs bundle energy $\tfrac12\|\Phi\|_{L^2}^2$ (which is 
sometimes called the
\emph{Hitchin function}) and the critical manifolds of this flow consist of Hodge bundles.
We show that the Hodge bundles correspond to a special class of minimal surfaces known as \emph{superminimal} surfaces, and 
we describe the conditions on the parameters $(W,Q_W,[\xi])$ which determine these.

The main structure of each component of the moduli space is carried by its fibre $\caM_l(\Sigma_c,\RH^4)$ over
$c\in\caT_g$. We show that for $l\neq 0$ this fibre is a complex vector bundle whose zero section is the submanifold 
$\caS_{c,l}$ of all superminimal immersions with $\chi(T\Sigma^\perp)=l$ and conformal structure $c$. Further, 
we show that this bundle is the downward gradient flow from $\caS_{c,l}$. 

The component $\caM_0(\Sigma,\RH^4)$ is slightly different because the corresponding manifold
$\caS_{c,0}$ parametrising superminimal immersions with flat normal bundle consists entirely of decomposable 
Hodge bundles, so these must be excluded from $\caM_0(\Sigma,\RH^4)$ but instead appears on its boundary. In fact, we show 
that a superminimal immersion with flat normal bundle must be totally geodesic into a copy of $\RH^2$. It follows that $\caS_{c,0}$ 
consists of multiple copies of $\caM(\Sigma,\RH^2)\simeq\caT_g$:
the multiplicity is caused by the freedom to choose how the corresponding reducible representations $\rho$ act in the normal bundle. 
Nevertheless, this component $\caM_0(\Sigma,\RH^4)$ does correspond to the downward gradient flow from $\caS_{c,0}$. In
addition to $\caS_{c,0}$, the boundary of $\caM_0(\Sigma,\RH^4)$
also contains one copy of $\caM(\Sigma,\RH^3)/\Z_2$, the moduli space of linearly full but unoriented minimal immersions into
$\RH^3$.

Finally, we observe that all these components $\caM_l(\Sigma,\RH^4)$ have a common ``boundary at infinity'' corresponding 
to the Higgs bundles
with zero Higgs fields: the absolute minima of the Hitchin function. Since the Hitchin function is essentially the area of
the minimal immersion (on a fundamental domain) this limit corresponds to collapsing the immersion to a constant map. 

By combining the results here with those of our earlier paper \cite{LofM15} we can make a reasonable conjecture for the
structure of $\caM(\Sigma,N)$ for an general rank one non-compact symmetric space. Namely, we expect it admits the complex 
structure of an open subvariety of a complex analytic family over
Teichm\"uller space. We expect the fibre over $[c]\in\caT_g$ is a disjoint union of vector bundles, 
each over a base which consists of the Hodge bundles in one critical manifold of the
Hitchin function. It is not hard to show that for $N=\CH^n$ the Hodge bundles always correspond to superminimal immersions
and we expect the same to hold for the other rank one symmetric spaces.
Given this, we conjecture that the connected components of $\caM(\Sigma,N)$ are
indexed by the topological invariants of linearly full superminimal immersions. It is
reasonable to expect to classify these for rank one symmetric spaces.

Outside the case of rank one symmetric spaces there is a great deal yet to be done.
Labourie conjectured in \cite{Lab06} that every \emph{Hitchin representation} into a split real form should admit a
unique minimal surface, which would imply that the moduli space of Hitchin representations (whose parametrisation is due to
Hitchin \cite{Hit92}) provides components of $\caM(\Sigma,N)$ when the isometry group $G$ of $N$ is a split real form.
Labourie recently proved his conjecture for split real forms of rank-two complex simple Lie groups \cite{Lab17}. Rank-two phenomena 
also appear a bit indirectly below as in harmonic sequence theory in Subsection \ref{superminimal-subsection} and the geometry of 
rank-two holomorphic vector bundles in Appendix \ref{appendix-b-section}. It should be quite interesting to develop the 
theory further in higher rank.

When $G=Sp(4,\R)$ the Hitchin representations are all \emph{maximal} (i.e., have maximal Toledo invariant) and
Collier \cite{Col} extended Labourie's result to cover all maximal representations in $Sp(4,\R)$. Very recently Collier and
collaborators have further extended this uniqueness result to maximal representations in $PSp(4,\R)$ \cite{AleC} and $SO(2,n)$
\cite{ColTT}.
Since $PSp(4,\R)\simeq SO_0(2,3)$ an adaptation of the techniques of this current paper may shed some light on the minimal
surfaces for non-maximal representations.
Finally, let us mention that Higgs bundles have also been used to study minimal surfaces by Baraglia \cite{Barthesis,Bar15}
(his concept of cyclic surfaces is central to Labourie's proof), while
Alessandrini \& Li characterised AdS 3-manifolds using minimal surfaces into the Klein quadric and Higgs bundles
for $G=SL(2,\mathbb R) \times SL(2,\mathbb R)$ \cite{AL15}. We also draw the reader's attention to the recent work
of Baraglia \& Schaposnik \cite{BarS}, which provides an interesting perspective on the structure of the regular
fibres of the Hitchin map for orthogonal Higgs bundles. Although we deal here exclusively with the non-regular fibre
$\tr\Phi^2=0$, it would be interesting to know how their extension bundle data behaves as one approaches this fibre.

\smallskip\noindent
\textbf{Acknowledgements.} This research was supported by the LMS Scheme 4 Research in Pairs grant \# 41532 and U.S. National Science 
Foundation grants DMS 1107452, 1107263, 1107367 ``RNMS:
GEometric structures And Representation varieties" (the GEAR Network).  The first author is also grateful to the Simons Foundation 
for partial support under Collaboration Grant for Mathematicians 210124.
Both authors are grateful to Laura Schaposnik and Steve Bradlow for informative conversations regarding orthogonal Higgs bundles.
The second author is grateful to the \textit{Centre for Quantum Geometry of Moduli Spaces}, University of Aarhus, for the
opportunity to visit in Sept.\ 2016 and discuss early stages of this work with Qiongling Li. In particular, \S 3 below is
based on observations
made by Qiongling Li and Daniele Alessandrini and we are grateful for their permission to use their unpublished ideas here.

\section{Equivariant minimal surfaces in $\RH^n$ and Higgs bundles.}

\subsection{Equivariant minimal surfaces.}
Let $\Sigma$ be a compact oriented surface of genus $g\geq 2$ and let $c:\pi_1\Sigma\to \Isom^+(\caD)$
be a Fuchsian representation into the group of orientation preserving isometries of the
Poincar\'{e} disc $\caD$. Let $\Sigma_c= \caD/c(\pi_1\Sigma)$ be the corresponding compact Riemann surface.
For any non-compact (and for simplicity, irreducible) globally symmetric space $N$, with
isometry group $G$, we say a minimal immersion $f:\caD\to N$ is equivariant with respect to a
representation $\rho:\pi_1\Sigma\to
G$ when $f\circ c(\delta)=\rho(\delta)\circ f$ for all
$\delta\in\pi_1\Sigma$. In general one wants to allow branched minimal immersions, which for this part of the
discussion we will do.
Naturally, we want to consider such triples $(f,c,\rho)$ to be equivalent when they are related by isometries
of the domain or codomain. Accordingly, we will write
\[
(f,c,\rho)\sim (f',c',\rho'),
\]
whenever there is $\gamma\in\Isom^+(\caD)$ and $g\in G$ for which
\[
c' = \gamma c\gamma^{-1},\quad \rho'=g\rho g^{-1}, \quad f' = gf_\gamma,
\]
where $f_\gamma(z) = f(\gamma^{-1}z)$.
The equivalence class will be written $[f,c,\rho]$.
The set of these equivalence classes for which $\rho$ is irreducible and $f$ is oriented
will be our \textit{moduli space of equivariant minimal surfaces in $N$}, which we will denote by $\caM(\Sigma,N)$.
We choose $\rho$ to be irreducible to avoid having a multiplicity of maps which only differ by changing a reductive factor
in the decomposition of $\rho$. For example, for every totally geodesic embedding of $\RH^2\simeq\caD$ in $\RH^n$ one can take
$\rho$ to be $c$ post-composed with an embedding of $\Isom^+(\caD)\simeq SO_0(2,1)$ into $SO_0(n,1)$, but one can also alter
this by any reductive representation of $\pi_1\Sigma$ into $SO(n-2)$ and acting in the normal bundle of the immersion. Such
behaviour turns up at the boundary of the moduli space and creates singularities and lower dimensional strata there.

The moduli space of Fuchsian representations up to conjugacy is Teichm\"{u}ller space $\caT_g$, while
the moduli space of reductive representations $\rho$ up to conjugacy is the character variety $\caR(\pi_1\Sigma,G)$.
The subset of irreducible representations is an open submanifold.
For a fixed conformal structure, standard uniqueness theorems for the harmonic metric (e.g.\ \cite{Cor,Don})
apply. It follows that we have an injective map
\begin{equation}\label{eq:F}
F:\caM(\Sigma,N)\to \caT_g\times \caR(\pi_1\Sigma,G), \quad [f,c,\rho]\mapsto ([c],[\rho]),
\end{equation}
where the square brackets denote conjugacy classes. The topology we will use for
$\caM(\Sigma,N)$ is the one induced by this injection. Indeed, we can use this to put a real analytic structure
on $\caM(\Sigma,N)$.
A conjugacy class $[c]\in\caT_g$ is sometimes called a \emph{marked conformal structure}:
we will denote the subset of equivariant minimal surfaces with fixed marked conformal structure
$[c]$ by $\caM(\Sigma_c,N)$.

Now recall the central result of non-abelian Hodge theory, which describes the relationship with Higgs bundles (see, for
example, \cite{Gar}).
This says that for each Fuchsian representation $c$ there is a
bijective correspondence  between polystable $G$-Higgs bundles over $\Sigma_c$ and
reductive representations $\rho:\pi_1\Sigma\to G$, up to their respective equivalence classes. This correspondence
gives a homeomorphism from the Higgs bundle moduli space $\caH(\Sigma_c,G)$
to $\caR(\pi_1\Sigma,G)$ which is real analytic away from singularities.

This central result works by assigning to each polystable Higgs bundle an equivariant harmonic map (to be precise, a triple
$[f,c,\rho]$ where $f$ is harmonic). It identifies
the Higgs field $\Phi$ with the restriction of $df$ to $T^{1,0}\caD$, which we will denote by $\partial f$. It is a well-known
fact that, using $g^\C$  to denote the complex bilinear extension of the metric $g$ of $N$,
$f$ is weakly conformal precisely
when $g^\C(\partial f,\partial f)=0$. Since the metric on $N$ comes from the Killing form on
$G$, it follows that this harmonic map $f$
is weakly conformal, and therefore a branched minimal immersion, precisely when the Higgs field $\Phi$ satisfies
$\tr(\ad\Phi^2)=0$. Therefore the map from $\caM(\Sigma_c,N)$ to the moduli space $\caH(\Sigma_c,G)$ of polystable
$G$-Higgs bundles, which assigns to each equivariant minimal surface its Higgs bundle data, is injective and its
image lies in the complex analytic subvariety given by the equation $\tr(\ad\Phi^2)=0$.
When $G$ is has real rank one (i.e., $N$ is a rank one symmetric space) this level set $\tr(\ad\Phi^2)=0$ is
the \emph{nilpotent cone}.

Now it is reasonable to expect that as $[c]$ varies over Teichm\"{u}ller space the Higgs bundle moduli spaces form a
\emph{complex analytic family} in the sense of Kodaira \& Spencer \cite{Kod}, and moreover that the
function $\tr(\ad\Phi^2)$ is holomorphic on this family. Given this, $\caM(\Sigma,N)$
would acquire the structure of a complex analytic family with fibres $\caM(\Sigma_c,N)$.
This has been shown for the case $N=\CH^2$ in \cite{LofM15} by focussing more directly on
the properties which characterise a $PU(2,1)$-Higgs bundle for which
$\tr(\ad\Phi^2)=0$.
\begin{rem}
It is very interesting to note that the mapping class group of $\Sigma$ acts naturally on $\caM(\Sigma,N)$. Recall that, by the
Dehn-Nielsen theorem, the mapping class group is isomorphic to the group $\Out(\pi_1\Sigma)=\Aut(\pi_1\Sigma)/\Inn(\pi_1\Sigma)$
of outer automorphisms of $\pi_1\Sigma$, where $\Inn(\pi_1\Sigma)$ is the subgroup of automorphisms given by conjugation.
For any $\tau\in\Aut(\pi_1\Sigma)$ and equivariant minimal surface $(f,c,\rho)$ it is easy to check that
$(f,\tau^*c,\tau^*\rho)$ is again an equivariant minimal surface and that the
equivalence class $[f,c,\rho]$ is unchanged when $\tau$ is an inner automorphism. Note that $\Out(\pi_1\Sigma)$ acts
similarly on $\caT_g\times \caR(\pi_1\Sigma, G)$ and that the embedding $F$ in \eqref{eq:F}
is equivariant with respect to this action.
\end{rem}

\subsection{Equivariant minimal surfaces in $\RH^n$.} We will now restrict our attention to the case where
$N=\RH^n$ and $G=SO_0(n,1)$, the connected component of the identity in $SO(n,1)$, for $n\geq 3$.
Our aim
here is to characterise the Higgs bundles which correspond to linearly full minimal immersions and show how
the Higgs bundle data relates to the metric and second fundamental form of the immersion.
Recall (from e.g., \cite{ApaG}) that an $SO_0(n,1)$-Higgs bundle over $\Sigma_c$ is uniquely determined by an
equivalence class of data $(V,Q_V,\varphi)$ where $V$ is a rank $n$ holomorphic vector bundle with a fixed isomorphism
$\det(V)\simeq 1$, $Q_V$ is an orthogonal
structure on $V$ and $\phi\in H^0(K\otimes\Hom(1,V))$. The Higgs bundle itself is the rank $n+1$ bundle $E=V\oplus 1$
with orthogonal structure
\begin{equation}\label{eq:Q_E}
Q_E = \begin{pmatrix} Q_V & 0 \\ 0 & -1\end{pmatrix},
\end{equation}
and Higgs field
\[
\Phi =\begin{pmatrix} 0 & \phi \\ \phi^t & 0 \end{pmatrix}
\]
where $\phi^t\in  H^0(K\otimes\Hom(V,1))$ is the dual of $\phi$ with respect to $Q_E$. Notice that we use a different
convention from \cite{ApaG} for the sign of the orthogonal structure on the trivial summand. This fits better with the the
interpretation of $Q_E$ as the complex bilinear extension of a Lorentz metric on $\R^{n,1}$ given below.
From now on we will simply refer to $(V,Q_V,\phi)$ as the $SO_0(n,1)$-Higgs bundle.

Now let us recall how the equivariant harmonic map $(f,c,\rho)$ is related to such a Higgs bundle.
Let $\C^{n,1}$ denote $\C^{n+1}$ as a pseudo-Hermitian vector space with inner product
\[
\h{v}{w} = v_1\bar w_1 +\ldots v_n\bar w_n-v_{n+1}\bar w_{n+1}.
\]
A pair $(c,\rho)\in\caT_g\times\caR(\pi_1\Sigma,G)$ provides an action of $\pi_1\Sigma$ on the trivial bundle
$\caD\times\C^{n,1}$ and the quotient is
a flat $SO_0(n,1)$-bundle $E=\caD\times_{\pi_1\Sigma}\C^{n,1}$ over $\Sigma_c$,
with pseudo-Hermitian metric $\h{\ }{\ }$, a real involution $\bar{\ }:E\to E$ and a
flat pseudo-Hermitian connexion $\nabla^E$. The first two equip $E$ with an
orthogonal structure $Q_E$ (given by $Q_E(\sigma,\sigma)= \h{\sigma}{\bar\sigma}$).
The connexion equips $E$ with a holomorphic structure $\bpartial_E=(\nabla^E)^{0,1}$ with
respect to which $Q_E$ is holomorphic.

From this the $SO_0(n,1)$-Higgs bundle $(V,Q_V,\phi)$ is obtained as follows.
We first identify $\RH^n$ with the pseudo-sphere
\[
 S^{n,1} = \{v\in\R^{n,1}:\h{v}{v}=-1,\ v_{n+1}>0\},
\]
in Minkowski space $\R^{n,1}$ so that we can consider $T\RH^n\subset
\RH^n\times\R^{n,1}$, with metric $g$ obtained by restriction of $\h{\ }{\ }$. We can consider
the quotient $V=(f^{-1}T^\C\RH^n)/\pi_1\Sigma$ as a subbundle of $E$. A choice of orientation of $S^{n,1}$ fixes
a trivial line subbundle $1$ such that $E = V\oplus 1$
and the decomposition is orthogonal. It follows that the orthogonal projection of
$\nabla^E$ onto $V$, which we will denote simply
by $\nabla$, coincides with the pullback of the Levi-Civita connexion on $\RH^n$.
Therefore the holomorphic structure $\bpartial_V$ of $V$ agrees with that induced by
$\nabla^{0,1}$.

The Higgs field carries the information of the differential of $f$. To be precise,
the differential of $f$ extends complex linearly to $df:T^\C\caD\to T^\C\RH^n $
and thus has a type decomposition
\[
df = \partial f +\bpartial f,\quad \partial f:T^{1,0}\caD\to T^\C\RH^n ,\quad \bpartial f:T^{0,1}\caD\to  T^\C\RH^n.
\]
By equivariance we can think of $\partial f$ as a section of $\caE^{1,0}(V)$. The harmonicity
of $f$ ensures that $\phi=\partial f$ is a holomorphic section of
$K\otimes\Hom(1,V)\simeq K\otimes V$.

In the reverse direction, when $(V,Q_V,\phi)$ is polystable and $\tr(\Phi^2)=0$ we obtain a weakly conformal
equivariant map. This map will be an immersion if and only if $\phi$ has no zeroes.
To understand more precisely the structure of $(V,Q_V,\phi)$ for a minimal immersion, we begin by characterising
those which have $\tr(\Phi^2)=0$ and $\phi$ nowhere vanishing.
\begin{lem}\label{lem:tr=0}
An $SO_0(n,1)$-Higgs bundle $(V,Q_V,\phi)$ satisfies $\tr(\Phi^2)=0$ precisely when $\phi^t\circ\phi=0$. In that case,
when $\phi$ has no zeroes $(V,Q_V,\phi)$ uniquely determines, and is determined by, a triple $(W,Q_W,[\xi])$ where $(W,Q_W)$ is a
holomorphic $SO(n-2,\C)$-bundle, $\xi\in H^1(\Hom(W,K^{-1}))$ and $[\xi]$ is its equivalence class under the
action of the group $SO(Q_W)$ of isomorphisms of $(W,Q_W)$.
\end{lem}
\begin{proof}
It is a simple exercise to show that $\tr(\Phi^2)=0$ if and only if $\phi^t\circ\phi=0$. Now let $V_{n-1}=\ker(\phi^t)$
and define $W = V_{n-1}/\im(\phi)$. Since $\phi$ has no zeroes this describes $V_{n-1}$ as an extension bundle
\begin{equation}\label{eq:phi}
0\to K^{-1}\stackrel{\phi}{\to}V_{n-1}\to W\to 0.
\end{equation}
This determines $W$ and an extension class $\xi\in H^1(\Hom(W,K^{-1}))$. Further, $\phi^t\circ\phi=0$ implies that
$\im(\phi)$ is $Q_V$-isotropic since $Q_V(\phi(Z),\phi(Z)) = Q_V(Z,\phi^t\circ\phi(Z))=0$. Hence $Q_V$ descends to an orthogonal
structure $Q_W$ on $W$. Isomorphisms of $(V,Q_V,\phi)$ act non-trivially on the extension class through isomorphisms of
$(W,Q_W)$, and therefore $\xi$ is only determined up to this action of $SO(Q_W)$. 

Conversely, given $(W,Q_W,[\xi])$ we model $V$ smoothly on $K^{-1}\oplus W\oplus K$ and give it the orthogonal structure
\begin{equation}\label{eq:Q_V}
Q_V=\begin{pmatrix} 0&0&1\\0&Q_W&0\\ 1&0&0\end{pmatrix},
\end{equation}
i.e., using $Q_W$ and the canonical pairing $K^{-1}\times K\to 1$. Then $(V,Q_V)$ has an $SO(n,\C)$-structure since there
is a fixed isomorphism $\det(V)\simeq\det(W)\simeq 1$. The Higgs field is determined by the inclusion
$\phi:K^{-1}\to V$. We claim $V$ has a unique holomorphic structure for which: (i)
$V_{n-1}$ is given up to isomorphism by $[\xi]$ as in \eqref{eq:phi}, (ii) $Q_V$ is holomorphic. To see
this, we observe that any holomorphic structure for which the flag $K^{-1}\subset V_{n-1}\subset V$ is holomorphic must
correspond to a $\bpartial$-operator of the form
\[
\bpartial_V = \begin{pmatrix} \bpartial & \alpha_1 & \alpha_2\\
0&\bpartial_W&\alpha_3\\
0&0&\bpartial
\end{pmatrix}
\]
where the matrix indicates how the operator acts on $V$ as the direct sum
$K^{-1}\oplus W\oplus K$. Here $\bpartial_W$ induces the holomorphic structure on $W$,
$\bpartial$ induces the holomorphic structure on both $K$ and $K^{-1}$, and $\alpha_j$ are $(0,1)$-forms
taking values in the appropriate bundle
homomorphisms (so $\alpha_1\in\caE^{0,1}(\Hom(W,K^{-1}))$ and so forth). A straightforward calculation shows that
for $Q_V$ to be holomorphic we must have $\alpha_1=-\alpha_3^t$ and
$\alpha_2=0$. Clearly $\alpha_1$ is a representative of $[\xi]$ in the sense that  
the cohomology  class class of $\alpha_1$ in $H^{0,1}(\Hom(W,K^{-1}))$ determines the extension \eqref{eq:phi}
via the Dolbeault isomorphism $H^{0,1}(\Hom(W,K^{-1}))\simeq H^1(\Hom(W,K^{-1}))$. 
Now suppose $\alpha_1'$ is any other representative of $[\xi]$, i.e, there exist $g\in SO(Q_W)$ such that $\alpha_1'$ and
$\alpha_1\circ g$ are cohomologous. Then there is $\psi\in\caE^0(\Hom(W,K^{-1}))$ such that
\[
\alpha_1' = \alpha_1\circ g +\bar\partial\psi-\psi\bar\partial_W,
\]
A direct calculation shows that the operator $\bar\partial_V'$ obtained from $\bar\partial_V$ by using $\alpha_1'$ in place 
of $\alpha_1$ is given by the gauge transformation $\bar\partial_V' = T^{-1}\bar\partial_V T$ where
\[
T = \begin{pmatrix} 1&0&0\\ 0&g&0\\0&0&1\end{pmatrix}
\begin{pmatrix} 1 & \psi &\theta\\ 0&1&-\psi^t\\ 0&0&1\end{pmatrix}
\]
and $\theta\in\caE^0(\Hom(K,K^{-1}))$ is given by $\theta=-\tfrac12 \psi\circ\psi^t$.
Further, this gauge transformation preserves $Q_V$ and the Higgs field.
Hence $(V,Q_V,\phi)$ is well-defined by the data $(W,Q_W,[\xi])$.
\end{proof}
We will also need to understand when $(V,Q_V,\phi)$ is polystable.
Fortunately there is a simple characterisation due to Aparicio \& Garc\'ia-Prada.
\begin{thm}[\cite{ApaG}, Prop 2.3 \& Thm 3.1]\label{thm:stab}
When $n>2$ an $SO_0(n,1)$ Higgs bundle is stable if for any isotropic subbundle
$W\subset V$ with $\phi^t(W)=0$ we have $\deg(W)<0$. It is polystable if it is
a direct sum of stable $G$-Higgs bundles where $G$ is either $SO_0(k,1)$, $SO(k)$ or $U(k)\subset SO(2k)$.
\end{thm}
\begin{rem}\label{rem:decomposable}
We will say that the Higgs bundle is \emph{decomposable} when it is a direct sum of more than one $G$-Higgs bundle.
Unlike the case of $G=GL(n,\C)$, stability does not imply indecomposability.
For example, it is shown in
\cite[Prop.\ 3.2]{ApaG} that if the decomposition above involves only Higgs subbundles for $SO_0(k,1)$ or $SO(l)$ with $l\neq 2$,
then $(V,Q_V,\phi)$ is stable. Clearly irreducible representations correspond precisely to indecomposable Higgs bundles. These
also provide smooth points in the moduli space of $SO_0(n,1)$-Higgs bundles (and hence in the character variety) by
Thm 5.5 and Cor 4.4 of \cite{ApaG}.
\end{rem}
The geometric meaning for $f$ of irreducibility of $\rho$ is the following. We say $f$ is \emph{linearly full}
if its image does not lie in a totally geodesic copy of $\RH^k$ in $\RH^n$ for some $k<n$.
\begin{lem}\label{lem:irred}
An equivariant minimal immersion $f:\caD\to \RH^n$ is linearly full if and only if the representation $\rho$ is irreducible.
\end{lem}
\begin{proof}
Clearly if $f$ is not linearly full its image lies in some copy of $S^{k,1}\subset \R^{k,1}\subset \R^{n,1}$, and this must
be preserved by $\rho$, hence $\rho$ is reducible. Conversely, suppose $\rho$ is reducible, then its Higgs bundle is
decomposable. Since the Higgs field is non-trivial there must be at least one (and therefore precisely one) indecomposable
subbundle with group $SO_0(k,1)$ for some $k<n$. The other Higgs bundle summands have compact group structures, and therefore
their Higgs fields are
trivial. Since the Higgs field represents $df$, the map $f$ takes values in the totally geodesic copy of $\RH^k$ which
corresponds to the $SO_0(k,1)$-Higgs summand.
\end{proof}
We can now describe more explicitly how the Higgs bundle data is related to the
classical minimal surface data, namely its induced metric $\gamma=f^*g$ and its second
fundamental form $\II$ (recall that $\II(X,Y) = (\nabla_XY)^\perp$ for $X,Y\in\Gamma(T\Sigma)$).
Given an immersion $f$ we have a smooth orthogonal decomposition
\begin{equation}\label{eq:normal}
V= T^\C\Sigma_c\oplus W
\end{equation}
where $W=(T^\perp\Sigma)^\C$. We will treat $\II$ as a $W$-valued complex bilinear form on $T^\C\Sigma$.
Then $f$ is minimal precisely when $\II^{1,1}=0$ in which case $\II$ is completely determined by $\II^{2,0}$.
By the Codazzi equations this is a holomorphic $W$-valued quadratic form on $\Sigma_c$
(see the appendix \ref{sec:GCR}). In the case $n=3$ it is essentially the Hopf differential (see \S \ref{sec:RH3} below).

Now we write $T^\C\Sigma_c=T^{1,0}\Sigma_c\oplus
T^{0,1}\Sigma_c$
and let $\gamma^\C$ denote the complex bilinear extension of the metric. This gives the orthogonal structure
on $T^\C\Sigma_c$, for which both $T^{1,0}\Sigma_c$ and $T^{0,1}\Sigma_c$ are isotropic. Let
$\hat\gamma:T^{0,1}\Sigma_c\to K$ denote the isomorphism $\bar Z\to \gamma^\C(\cdot,\bar Z)$. It has inverse
\begin{equation}\label{eq:hatgamma}
\hat\gamma^{-1}:K\to T^{0,1}\Sigma_c;\quad dz\mapsto \bar Z/\|\bar Z\|^2_\gamma,
\end{equation}
whenever $dz(Z)=1$. In particular, we obtain an isomorphism $T^\C\Sigma\simeq K^{-1}\oplus K$ for which the
orthogonal structure
$\gamma^\C$ makes both $K^{-1},K$ isotropic and pairs them canonically. Thus we have a smooth isomorphism
\begin{equation}\label{eq:Vdecomp}
V=(f^{-1}T\RH^n)^\C\simeq K^{-1}\oplus W\oplus K.
\end{equation}
The complex bilinear extension of the metric on $T\RH^n$ provides the orthogonal structure $Q_W$ on $W$.
Let $\bpartial$ to denote the
holomorphic structure on
both $K^{-1}$ and $K$, and let $\bpartial_W$ denote the holomorphic structure on
$W$ coming from the connexion in the normal bundle.
Using the isomorphism $\hat\gamma$ above we associate $\II$ to a
$\Hom(K,W)$-valued $(0,1)$ form $\beta$, defined locally by
\begin{equation}\label{eq:beta}
\beta(\bar Z):dz\mapsto \II(\bar Z,\bar Z)/\|\bar Z\|_\gamma^2,
\end{equation}
for $Z=\partial/\partial z$.
\begin{thm}\label{thm:Higgs}
Let $[f,c,\rho]$ be an equivariant minimal surface in $\RH^n$, $n\geq 3$, with Higgs bundle $(V,Q_V,\phi)$. Then
$(V,Q_V,\phi)$ is given by the data $(W,Q_W,[-\beta^t])$ as in Lemma \ref{lem:tr=0}, where $W$ is the complexified normal
bundle, $Q_W$ is the complex bilinear extension of the normal bundle metric,
$\beta^t$ is the adjoint of $\beta\in\caE^{0,1}_{\Sigma_c}(\Hom(K,W))$ with respect to $Q_V$, and $[-\beta^t]\in
H^1(\Hom(W,K^{-1}))/SO(Q_W)$ is obtained via the Dolbeault isomorphism.

Conversely, if the Higgs bundle determined by $(W,Q_W,[\xi])$ is stable and indecomposable
then it determines a unique linearly full equivariant minimal immersion $[f,c,\rho]$.
\end{thm}
\begin{proof}
Recall that with respect to the orthogonal decomposition \eqref{eq:normal}
the Levi-Civita connexion can be block-decomposed as
\[
\nabla = \begin{pmatrix}
\nabla^\Sigma & -B^t\\ B & \nabla^\perp
\end{pmatrix}
\]
where $\nabla^\Sigma,\nabla^\perp$ denote respectively the induced connexions on the tangent bundle and the normal bundle, and
$B\in\caE^1_\Sigma(\Hom(T^\C\Sigma,W))$ represents the second fundamental form, i.e.,
$B(X):Y\to \II(X,Y)$. Its adjoint $B^t$ is with respect to the metric on $f^{-1}T\RH^n$.

It follows that
with respect to the decomposition \eqref{eq:Vdecomp} we can write
\[
\bpartial_V = \begin{pmatrix}
\bpartial & -\beta^t & 0\\ \alpha &\bpartial_W & \beta\\ 0&-\alpha^t &\bpartial
\end{pmatrix},\qquad
Q_V=\begin{pmatrix} 0&0&1\\0&Q_W&0\\1&0&0\end{pmatrix},
\]
where $\alpha\in\caE^{0,1}_{\Sigma_c}(\Hom(K^{-1},W))$ and $\beta\in\caE^{0,1}_{\Sigma_c}(\Hom(K,W))$ are
given locally by
\[
\alpha(\bar Z):Z\to \II(Z,\bar Z),\quad \beta(\bar Z):dz\to \II(\bar Z,\bar Z)/\|\bar
Z\|_\gamma^2.
\]
However, $\II(Z,\bar Z)=0$ since $f$ is minimal, hence $\bpartial_V$ is given by
\begin{equation}\label{eq:bpartialV}
\bpartial_V =
\begin{pmatrix}
\bpartial& -\beta^t & 0\\ 0&\bpartial_W &\beta \\ 0&0&\bpartial
\end{pmatrix},
\end{equation}
From the proof of Lemma \ref{lem:tr=0} it follows that $\xi=[-\beta^t]$.

The converse is just non-abelian Hodge theory together with Lemmas \ref{lem:tr=0} and \ref{lem:irred}.
\end{proof}
Note that the isomorphism $W\simeq W^*$ induced by $Q_W$ allows us to identify $\Hom(W,K^{-1})$ with
$\Hom(K,W)$, and this identifies $\beta^t$ with $\beta$. From now one we will make this identification.

Regarding the case where $[\beta]=0$ we make the following observation (cf. Theorem 3.1 of \cite{LofM15}).
\begin{prop}\label{prop:beta=0}
The class $[\beta]$ above is trivial if and only if $\II$ is identically zero, i.e., if and only if
the map $f$ is a totally geodesic embedding.
\end{prop}
In particular, the case where $[\beta]=0$ corresponds precisely to those triples $[f,c,\rho]$ for which the representation $\rho$
factors through an embedding of $SO_0(2,1)$ into $SO_0(n,1)$. In this case $\rho$ is reducible: it is just the Fuchsian
representation $c$ post-composed with this embedding.
\begin{proof}
It suffices to show that there is a Hermitian metric $h$ on $\Hom(K,W)$ with respect to which $\beta$ is harmonic.
This metric is none other than the metric induced by the immersion, for with respect to that metric the Hodge star
\[
\caE^{0,1}(K^{-1}\otimes W)\stackrel{\bar\star}{\to}\caE^{1,0}(W^*\otimes K)
\]
maps $\beta(\bar Z)d\bz$ to $-ih(\cdot,\beta(\bar Z))dz$. Now using \eqref{eq:beta} and the definition of $Q_W$ we see that
\[
h(\cdot,\beta(\bar Z)) = Q_W(\cdot, \II(Z,Z)).
\]
But $\II(Z,Z)$ is holomorphic, and therefore $\bar\partial\bar\star \beta=0$. Hence $\beta$ is harmonic.
\end{proof}

Lemma \ref{lem:tr=0} and Theorem \ref{thm:Higgs} show we should be able to parametrise equivariant minimal immersions by their data
$([c],W,Q_W,[\beta])$. The
major difficulty in general is to identify in a satisfying way the conditions which ensure stability of the Higgs bundle.
In the cases where $n=3$ or $n=4$ the bundle $(W,Q_W)$ is simple enough that we can do this, and this is the purpose of the
remainder of the article.

\section{Equivariant minimal surfaces in $\RH^3$.}\label{sec:RH3}
Theorem \ref{thm:Higgs} leads very quickly to a description of the moduli space $\caM(\Sigma,\RH^3)$.
For in this case $f^{-1}T\RH^3/\pi_1\Sigma$ is $SO(3)$-bundle with decomposition into $T\Sigma\oplus
T\Sigma^\perp$. Since $f$ is an oriented immersion $T\Sigma^\perp$ is trivial. Therefore, as a smooth bundle,
\[
V \simeq K^{-1}\oplus 1\oplus K,
\]
with orthogonal structure
\[
Q_V = \begin{pmatrix} 0&0&1\\0&1&0\\ 1&0&0\end{pmatrix}.
\]
It is well-known that $\rho$ has associated with it a $\Z_2$ invariant, which we will denote by $w_2(\rho)$,
equal to the second Steifel-Whitney
class of the $SO(3)$-bundle associated to $(V,Q_V)$ (see, for example, \cite{ApaG}). In our case this $SO(3)$-bundle is the
bundle of oriented frames of $f^{-1}T\RH^3/\pi_1\Sigma$,
and the question is whether or not this lifts to a $\Spin(3)$-bundle.  The decomposition $T\Sigma\oplus T\Sigma^\perp$
gives a reduction of structure group to $SO(2)$, and therefore it does lift because the first Chern class of $T\Sigma$
is even. Hence $w_2(\rho)=0$.
\begin{rem}
If we had allowed the possibility that $f$ has branch points then these would occur on a divisor $D\subset\Sigma_c$, and the
splitting of $V$ would replace $K^{-1}$ by $K^{-1}(D)$. In that case
$w_2(\rho)=\deg(K^{-1}(D))\bmod 2=\deg(D)\bmod 2$.
\end{rem}

Let $\nu$ be the unit normal field compatible with the orientation of $f$ and recall that the \emph{Hopf differential} 
of a minimal surface in a
$3$-manifold is the quadratic holomorphic differental $q = \h{\II^{2,0}}{\nu}$. In that case the quantity
$\beta$ in Theorem \ref{thm:Higgs} can be written as $\bar q\otimes \hat\gamma^{-1}$.
Here we interpret $\hat\gamma\in \Gamma(\Sigma_c,K\bar K)$ (cf.\ \eqref{eq:hatgamma}).
Thus we can write the holomorphic structure for $V$ in the form
\begin{equation}\label{eq:V3}
\bpartial_V = \begin{pmatrix}
\bpartial & -\bar q\otimes\hat\gamma^{-1} & 0\\
0 & \bpartial & \bar q\otimes\hat\gamma^{-1}\\
0&0&\bpartial \end{pmatrix}.
\end{equation}
and the holomorphic structure
depends only upon
\[
[\bar q\otimes\hat\gamma^{-1}]\in H^{0,1}(\Sigma_c,K^{-1})\simeq H^1(\Sigma_c,K^{-1}).
\]
In particular, $SO(Q_W)\simeq 1$.
By Proposition \ref{prop:beta=0} this cohomology class is trivial if and only if $q=0$.

With respect to the smooth isomorphism
\[
E\simeq (K^{-1}\oplus 1\oplus K)\oplus 1
\]
the Higgs field $\Phi$ and orthogonal structure $Q_E$ have the form
\[
\Phi = \begin{pmatrix} 0&0&0&1\\0&0&0&0\\0&0&0&0\\0&0&1&0\end{pmatrix},\quad
Q_E = \begin{pmatrix} 0&0&1&0\\0&1&0&0\\ 1&0&0&0\\0&0&0&-1\end{pmatrix}.
\]
It is not hard to check that this Higgs bundle $(V\oplus 1,\Phi)$ is stable if and
only if $q\neq 0$, and that when $q=0$ it is polystable with decomposition
\[
(V_2\oplus 1,\Phi')\oplus (1,0),
\]
where $V_2= K^{-1}\oplus K$ and $\Phi'$ is obtained from $\Phi$ by striking out the
second row and second column.

Now non-abelian Hodge theory provides the converse: to every equivalence class of data
$(c,\xi)$, where $c$ is a marked conformal structure and $\xi\in H^1(\Sigma_c,K^{-1})$, we obtain an equivariant minimal
surface $[f,c,\rho]$ in $\RH^3$. This is determined only up to the equivalence of such triples above.
As a consequence we can equip the moduli space $\caM(\Sigma,\RH^3)$
with the structure of a complex manifold. Since
$H^1(\Sigma_c,K^{-1})$ is the tangent space to Teichm\"{u}ller space $\caT_g$ at $c$, we deduce the following.
\begin{thm}\label{thm:RH3}
The moduli space of equivariant linearly full minimal immersions into $\RH^3$
can be identified with the bundle of punctured tangent spaces over Teichm\"{u}ller space,
\[
\caM(\Sigma,\RH^3)\simeq \{X\in T\caT_g:X\neq 0\}.
\]
In particular, this gives it the structure of a non-singular connected complex manifold of complex dimension $6(g-1)$.
\end{thm}
Note that when $\caM(\Sigma,\RH^3)$ is completed by totally geodesic immersions ($q=0$), which comprise all the non-full minimal
immersions in this case, we obtain the full tangent space to Teichm\"{u}ller space. We will denote this completed
space by $\ocaM(\Sigma,\RH^3)$.
\begin{rem}
A generalization of Theorem \ref{thm:RH3} in which branch points are allowed was already known to Alessandrini \& Li \cite{AL}
using a slightly different approach. With that generality the problem is equivalent to finding a parametrisation of the
components of the nilpotent cone. This was essentially done many years earlier by Donagi, Ein \& Lazarsfeld in \cite{DEL},
albeit with some translation required to adapt their results for $GL(2,\C)$-Higgs bundles to $PSL(2,\C)$-Higgs bundles.
Recall that the non-abelian Hodge theory for this case is due entirely to Hitchin \cite{Hit87} and
Donaldson \cite{Don}.
\end{rem}
\begin{rem}\label{rem:n=3}
Taubes studied a similar moduli space to $\ocaM(\Sigma,\RH^3)$ in \cite{Tau}, which he called the moduli space $\caH$ of
\textit{minimal hyperbolic germs}. Each of these is a pair $(\gamma,\II)$ consisting of a metric on $\Sigma$ and a trace-free
symmetric bilinear form $\II$ which together satisfy the Gauss-Codazzi equations for a minimal immersion into a $3$-space of constant
negative curvature $-1$ (it is easy to show that these are the only two pieces of information needed for these equations).
He showed that his moduli space is smooth and has real dimension $12(g-1)$, the same as the real dimension of $T\caT_g$. Now,
the pair $(\gamma,\II)$ is all that is needed to construct the Higgs bundles above, and the Gauss-Codazzi equations are exactly
the zero curvature equations for the related connexion $\nabla^E$ (see appendix \ref{sec:GCR} below). It follows that there is
a bijection between $\ocaM(\Sigma,\RH^3)$ and $\caH$ which assigns to each equivariant minimal immersion its
metric and second fundamental form. Indeed, Taubes shows that there is a smooth map
$\caH\to \caR(\pi_1\Sigma,SO(3,\C))$, given by mapping $(\gamma,\II)$ to the (conjugacy class of the) holonomy of a flat connexion
which essentially plays the role of $\nabla^E$ after the isomorphism between $SO_0(3,1)$ and $SO(3,\C)$. Hence there is a
smooth map
\[
\caH \to \caT_g\times \caR(\pi_1\Sigma,SO(3,\C)),
\]
by taking the conformal class of $\gamma$ for the first factor. It should be possible to use Taubes' calculations to show
that the image can be smoothly identified with $\ocaM(\Sigma,\RH^3)$.
\end{rem}

\section{Equivariant minimal surfaces in $\RH^4$.}
For $n=4$ the normal bundle $T\Sigma^\perp$ is an $SO(2)$-bundle, and therefore comes with a canonical complex structure $J$
compatible with the orientation. Locally this is given by
\[
J\nu_1=\nu_2,\quad J\nu_2=-\nu_1,
\]
with respect to an oriented orthonormal local frame $\nu_1,\nu_2$ for the normal bundle.
It follows that the complexified normal bundle splits into a direct sum of line subbundles given by the eigenspaces of $J$.
Let $L\subset W$ be the line subbundle for eigenvalue $i$, then the eigenbundle for $-i$ is $\bar L$ which is smoothly
isomorphic to the dual $L^{-1}$ of $L$.
Therefore $W\simeq L\oplus L^{-1}$, and this is an orthogonal decomposition since $J$
is an isometry on the normal bundle. In fact this isomorphism is holomorphic, since $J$ is
parallel for the normal bundle connexion $\nabla^\perp$, from which $W$ gets its holomorphic structure, and 
these line subbundles are $\nabla^\perp$-invariant. 

Since $(T\Sigma^\perp,J)\simeq L$ as a complex line bundle, the Euler number of
$T\Sigma^\perp$ satisfies
\begin{equation}\label{eq:degL}
\chi(T\Sigma^\perp) = \deg(L).
\end{equation}
Note also that $Q_W$ is just the canonical pairing $L\times L^{-1}\to 1$.
The projections of $W$ onto $L$ and $L^{-1}$ given, respectively, by
\[
\sigma\mapsto \tfrac12(\sigma -iJ\sigma),\quad \sigma\mapsto \tfrac12(\sigma +iJ\sigma),
\]
are consequently holomorphic. In particular, define
\begin{equation}\label{eq:theta}
\theta_2= \tfrac12(\II^{2,0} -iJ\II^{2,0}),\quad \theta_1=\tfrac12(\II^{2,0} +iJ\II^{2,0}).
\end{equation}
These are holomorphic sections of $K^2L$ and $K^2L^{-1}$ respectively.

From Theorem \ref{thm:Higgs} the holomorphic structure on $V$ is
determined by the cohomology class of $\beta\in \caE^{0,1}(\Hom(K,W))$ given by \eqref{eq:beta}.
Since $W\simeq L\oplus L^{-1}$ holomorphically, we can
represent the holomorphic structure $\bpartial_V$ with respect to the smooth isomorphism
\begin{equation}\label{eq:V4decomp}
V\simeq K^{-1}\oplus L \oplus L^{-1}\oplus K,
\end{equation}
by
\begin{equation}\label{eq:V4hol}
\bpartial_V = \begin{pmatrix}
\bpartial & -\beta_2 & -\beta_1 & 0\\
0 & \bpartial_1 & 0&\beta_1\\
0 &0 & \bpartial_2 &\beta_2\\
0 &0 &0& \bpartial\\
\end{pmatrix},
\end{equation}
where $\bpartial_1,\bpartial_2$ are the holomorphic structures on $L,L^{-1}$ and $\beta_1,\beta_2$ are the components
of $\beta$ with respect to the splitting $W\simeq L\oplus L^{-1}$.
In particular,
\[
\beta_j(\bar Z):dz \mapsto \overline{\theta_j(Z,Z)}/\|Z\|^2_\gamma,
\]
with $\beta_1(\bar Z)$ taking values in $K^{-1}L$ and $\beta_2(\bar Z)$ taking values in $K^{-1}L^{-1}$.
The forms $-\beta_1,-\beta_2$ determine extension bundles
\[
0\to K^{-1}\to V_1\to L^{-1},\quad 0\to K^{-1}\to V_2\to L,
\]
respectively, which are holomorphic sububundles of $V$. It is easy to see that the only
$Q_V$-isotropic subbundles of $V$ are $V_1,V_2$ or subbundles of these. Now we note that since $(W,Q_W)$ are
represented by
\[
\bar\partial_W = \begin{pmatrix} \bar\partial_1 &0\\ 0& \bar\partial_2\end{pmatrix},\quad 
Q_W=\begin{pmatrix} 0&1\\1&0\end{pmatrix},
\]
the group $SO(Q_W)\simeq \Ct$ acts on $(\beta_1,\beta_2)$ by 
\begin{equation}\label{eq:t-action}
a\cdot(\beta_1,\beta_2) = (a\beta_1,a^{-1}\beta_2),\quad a\in\Ct.
\end{equation}
We will use $[\beta_j]$ to denote the extension class of $\beta_j$, but $[\beta]$ to denote the
element of
\[
(H^1(\Sigma_c,K^{-1}L)\oplus H^1(\Sigma_c,K^{-1}L^{-1}))/\Ct,
\]
which represents this $\Ct$-orbit of the pair $([\beta_1],[\beta_2])$.

Let us now consider the conditions under which a minimal immersion can be constructed from
an $SO_0(4,1)$-Higgs bundle of the type just described.
That is, we fix $L$ and $[\beta]$
and construct $(V,\bpartial_V,Q_V,\phi)$ as above to obtain an $SO_0(4,1)$-Higgs bundle $(E,\Phi)$.
We need to ascertain when this bundle is stable and indecomposable.
For the remainder of this section we assume that $\deg(L)\geq 0$.
\begin{lem}\label{lem:stab1}
Suppose that $\deg(L)\geq 1$. Then $(V,\bpartial_V,Q_V,\phi)$ gives a stable $SO_0(4,1)$-Higgs bundle if and only if
$\deg(L)<2(g-1)$ and every line subbundle of $V_2$ has negative degree. All such Higgs bundles are indecomposable.
In particular, $V_2$ is a non-trivial extension, i.e., $[\beta_2]\neq 0$.
\end{lem}
\begin{proof}
According to the stability condition in Theorem \ref{thm:stab}, we must ensure that every isotropic subbundle
of $\ker(\phi^t)$ has negative degree. As a smooth bundle $\ker(\phi^t)=K^{-1} \oplus L\oplus L^{-1}$ and the
isotropic subbundles of this are therefore $V_1,V_2$ and their subbundles. By the stability criteria of 
Theorem \ref{thm:stab} we require these to all have negative degree. Since 
$\deg(L)\geq 1$, $V_1,V_2$ both have negative degree if and only if $\deg(L)<\deg(K)$. 
Now for any extension bundle of line bundles
\[
0\to \caF_1\to \caF\to \caF_2\to 0
\]
if $\deg(\caF_2)>\deg(\caF_1)$ then every holomorphic line subbundle
has degree no greater than $\deg(\caF_2)$. Moreover there is a line subbundle of degree equal to $\deg(\caF_2)$ if and only
if the extension is trivial. 
Hence line subbundles of $V_1$ necessarily have negative degree when $\deg(L)>0$.
\end{proof}
\begin{lem}\label{lem:stab2}
Suppose that $\deg(L)=0$. Then $(V,\bpartial_V,Q_V,\phi)$ gives a stable $SO_0(4,1)$-Higgs bundle if and only if
both $[\beta_1],[\beta_2]$ are non-zero. It is also indecomposable except when $L\simeq 1$ and $[\beta_1]=[\beta_2]$.
\end{lem}
\begin{proof}
That $[\beta_1]$, $[\beta_2]$ are both non-zero follows from similar reasoning to the proof of the previous lemma. Now
suppose $(E,\Phi)$ is decomposable. From Lemma \ref{lem:irred} $f$ maps into either a copy of $\RH^2$ or $\RH^3$.
Since the extensions $V_1,V_2$ are non-trivial it must be the latter. Thus
\[
(E,\Phi) = (V'\oplus 1,\Phi')\oplus (V'',0),
\]
where the first summand is a stable $SO_0(3,1)$-Higgs bundle and $V''\simeq 1$. Hence $W$ is trivial and therefore so is $L$.
In particular, the normal bundle $T\Sigma^\perp$ has a global  orthonormal frame $\nu_1,\nu_2$
for which $\nu_1$ is a unit normal to $f$ inside the tangent space to this $\RH^3$, and $\nu_2$ is a unit normal to the
$\RH^3$, so that $\nu_2$ is parallel. It follows that
\[
\h{\II(X,Y)}{\nu_2} = \h{\nabla_XY}{\nu_2} = -\h{Y}{\nabla_X\nu_2}=0.
\]
From \eqref{eq:theta} we have in general, locally,
\begin{equation}\label{eq:thetaj}
\theta_1(Z,Z) = \tfrac12(A_1-iA_2)(\nu_1+i\nu_2)dz^2,\quad \theta_2(Z,Z) = \tfrac12(A_1+iA_2)(\nu_1-i\nu_2)dz^2,
\end{equation}
where $A_j= \h{\II(Z,Z)}{\nu_j}$. So when $A_2=0$ we have
\[
\theta_1(Z,Z) = \tfrac12 A_1(\nu_1+i\nu_2)dz^2,\quad \theta_2(Z,Z) = \tfrac12 A_1(\nu_1-i\nu_2)dz^2.
\]
Now $L\simeq 1\simeq L^{-1}$ and the isomorphism identifies $\nu_1+i\nu_2$ with $\nu_1-i\nu_2$, therefore $\theta_1$ is
identified with $\theta_2$, hence $\beta_1$ is identified with $\beta_2$.
\end{proof}
To parametrise the space of pairs $([\beta_1],[\beta_2])$ which correspond to these stability conditions we need to introduce
some new spaces. For a fixed complex structure $\Sigma_c$ and integer $l$ satisfying $2(1-g)<l<2(g-1)$, let $\caV_{c,l}$ denote
the holomorphic vector bundle over $\Pic_l(\Sigma_c)$ (the moduli space of degree $l$ line bundles over $\Sigma_c$) whose fibre over
$L\in\Pic_l(\Sigma_c)$ is
$H^1(\Sigma_c,K^{-1}L^{-1})$. It is easy to check that, since $\deg(K^{-1}L^{-1})<0$, each fibre has positive dimension
\begin{equation}\label{eq:h1}
h^1(\Sigma_c,K^{-1}L^{-1})= 3(g-1)+l.
\end{equation}
This dimension is independent of the choice of $L$ and therefore Grauert's result \cite[section 10.5]{GR}
ensures that we do obtain a
holomorphic vector bundle. To each $\xi\in \caV_{c,l}$ we assign the extension bundle $0\to K^{-1}\to\xi\to L\to 0$ characterised
by it. Every line subbundle of $\xi$ has degree bounded above by $\deg(L)$, so there is a well-defined integer function
\[
\mu:\caV_{c,l}\to\Z,\quad \mu(\xi) = \max\{\deg(\lambda):\lambda\subset \xi\text{ a holomorphic line subbundle}\}.
\]
\begin{lem}\label{lem:maxdeg}
For $1\leq l <2(g-1)$ the set $\caV_{c,l}^0 = \{\xi\in\caV_{c,l}:\mu(\xi)<0\}$ is a non-empty Zariski open subvariety.
\end{lem}
Clearly $\caV_{c,l}^0$ is preserved by the scaling action of $\Ct$ since $\mu(\xi)$ depends only on the isomorphism class
of the bundle $\xi$. The proof of this lemma follows from Prop.\ 1.1 in \cite{LanN}, but we defer this to an appendix to 
avoid digression.

Now for $1\leq l <2(g-1)$ let $\caM(\Sigma_c,\RH^4)_l$ denote the moduli space of equivariant minimal immersions
for fixed Fuchsian representation $c$ and whose normal bundle has Euler number $l$. By Lemma \ref{lem:stab1} 
this set is parametrised by the variety
\[
\caW_{c,l}=\{(L,\xi_1,\xi_2)\in\iota^*\caV_{c,-l}\oplus\caV_{c,l}: \xi_2\in \caV_{c,l}^0\}/\Ct,
\]
where $\iota:\Pic_l(\Sigma_c)\to\Pic_{-l}(\Sigma_c)$ maps $L$ to $L^{-1}$ and the action of $\Ct$ is
\begin{equation}\label{eq:Ct}
a\cdot (L,\xi_1,\xi_2) = (L,a\xi_1,a^{-1}\xi_2).
\end{equation}
This action is free since $\xi_2\neq 0$, hence the quotient is non-singular. For $2(1-g)<l<0$ define
\[
\caW_{c,l}=\{(L,\xi_1,\xi_2)\in\iota^*\caV_{c,-l}\oplus\caV_{c,l}: \xi_1\in \iota^*\caV_{c,-l}^0\}/\Ct.
\]
This variety parametrises $\caM(\Sigma_c,\RH^4)_l$ for $l<1$. 
Finally, let $\caV_{c,0}^+$ denote the bundle $\caV_{c,0}$ without its zero section. By Lemma \ref{lem:stab2} the variety
\[
\caW_{c,0} = \{(L,\xi_1,\xi_2)\in \iota^*\caV_{c,0}^+\oplus\caV_{c,0}^+,\ \C.\xi_1\neq \C.\xi_2\}/\Ct,
\]
parametrises $\caM(\Sigma_c,\RH^4)_0$. 

Now consider the variation of $c$ over Teichm\"{u}ller space $\caT_g$.
By the same reasoning as in \cite[\S 6]{LofM15}, each family
\[
\caV_l=\cup_{c\in\caT_g} \caV_{c,l},
\]
is a complex analytic family over $\caT_g$.
Each $\caV_l$ is clearly a connected complex manifold, and by \eqref{eq:h1} each has dimension $10g-9$. 
Let $\caW_l$ denote the corresponding family over $\caT_g$ with
fibres $\caW_{c,l}$.
Then $\caW_l$ must also be connected, since each fibre is the quotient of a Zariski open subvariety in a vector space,
and of dimension $10g-10$.

In summary, for $N=\RH^4$ we have a bijective map
\begin{equation}\label{eq:caF}
\caF:\bigcup_{|l|<2(g-1)}\caW_l\to \caM(\Sigma,N),
\end{equation}
which assigns to each point $(c,L,[\xi])$ an indecomposable $SO_0(4,1)$-Higgs bundle, and therefore a
point $[f,c,\rho]$ in $\caM(\Sigma,N)$. We claim that $\caF$ is a diffeomorphism, and therefore the connected components of
$\caM(\Sigma,N)$ are given by $\caF(\caW_l)$. To prove this, let $\caF_l$ denote the restriction of $\caF$ to $\caW_l$.
\begin{thm}\label{thm:RH4}
Set $N=\RH^4$. For each $2(1-g)<l<2(g-1)$ the map $\caF_l:\caW_l\to \caM(\Sigma,N)$ is an injective local diffeomorphism. Hence 
we can smoothly identify $\caM(\Sigma,N)$ with the disjoint union $\cup_l \caW_l$. Consequently, $\caM(\Sigma,N)$ can be given the
structure of a non-singular complex manifold of dimension $10(g-1)$ with $4g-5$
connected components, indexed by the Euler number of the normal bundle $\chi(T\Sigma^\perp)$ with $|\chi(T\Sigma^\perp)|<2(g-1)$.
\end{thm}
We can relate this structure to the Morse theoretic study of the moduli space of $SO_0(n,1)$-Higgs bundles in \cite{ApaG}. This
is done in \S \ref{sec:Morse} below, after we have identified which equivariant minimal immersions correspond to
Hodge bundles.
\begin{proof}
By construction we have taken the smooth structure of $\caM(\Sigma,N)$ from its inclusion in $\caT_g\times
\caR(\pi_1\Sigma,G)$. We know each $\caF_l$ is injective, so it suffices to show that $\ker(d\caF_l)$ is injective at each point.
By post-composing $\caF_l$ with the inclusion we can write $\caF_l=(\pi_l,\psi_l)$ where
\[
\pi_l:\caW_l\to\caT_g,
\]
is simply the fibration of $\caW_l$ over $\caT_g$, and
\[
\psi_l:\caW_l\to \caR(\pi_1\Sigma,G)\simeq \caH(\Sigma_c,G),
\]
assigns the Higgs bundle for fixed choice of $c$. It follows that $d\caF_l$ is injective precisely when $d\psi_l$ is injective
on the tangent spaces to the fibres $\caW_{c,l}$ of $\pi_l$. Therefore it suffices to show that for fixed $c,l$ the map
\begin{equation}\label{eq:psi}
\psi:\caW_{c,l}\to \caH(\Sigma_c,G),
\end{equation}
is an immersion at each point. So fix a point $(L,[\beta])\in \caW_{c,l}$ and a $\bpartial$-operator $\bpartial_L$ on $L$
which induces its holomorphic structure. Let
\[
\hat\caW_{c,l} = \caE^{0,1}(\C)\times \caE^{0,1}(LK^{-1})\times\caE^{0,1}(L^{-1}K^{-1}),
\]
and define
\[
\hat\psi:\hat\caW_{c,l}\to \caH(\Sigma_c,G),\quad (\alpha,\eta_1,\eta_2)\mapsto \psi(L_\alpha,[\eta]_\alpha),
\]
where $L_\alpha$ is $L$ with the holomorphic structure $\bpartial_L+\alpha$ and 
\[
[\eta]_\alpha= \{(a[\eta_1],a^{-1}[\eta_2]):a\in\Ct\},
\]
where $[\eta_j]$ denotes the Dolbeault cohomology class of $\eta_j$ with respect to $L_\alpha$. Now fix a representative
$(0,\beta_1,\beta_2)\in\hat\caW_{c,l}$ for $(L,[\beta])$. 
Clearly $\psi$ will be an immersion at $(L,[\beta])$ if
\begin{equation}\label{eq:dpsi}
\frac{d}{dt}\hat\psi(t\alpha,\beta_1+t\eta_1,\beta_2+t\eta_2) =0
\end{equation}
implies $(\alpha,\eta_1,\eta_2)$ is tangent to the fibre of the projection $\hat\caW_{c,l}\to \caW_{c,l}$ at the point
$(0,\beta_1,\beta_2)$.  Now let $\bpartial_V$ denote the operator in
\eqref{eq:V4hol}, then $\hat\psi(t\alpha,\beta_1+t\eta_1,\beta_2+t\eta_2)$ is the Higgs bundle determined by $\bpartial_V+tA$
where
\[
A=\begin{pmatrix}
0& -\eta_2&-\eta_1&0\\
0&\alpha &0&\eta_1\\
0&0&-\alpha&\eta_2\\
0&0&0&0\end{pmatrix},
\]
and neither the Higgs field nor the orthogonal structure $Q_V$ depend upon $t$ (the inclusion $K^{-1}\to V$ is
holomorphic for every $t$). Therefore the left hand side of \eqref{eq:dpsi}
is represented by the equivalence class of the $\End(V)$-valued $(0,1)$-form $A$
with respect to infinitesimal gauge transformations. This class is trivial if there is a curve of gauge transformations
$g(t)\in\Gamma(\End(V))$ which is $Q_V$-orthogonal, transforms operators of the shape \eqref{eq:V4hol} into operators of
the same shape, and satisfies
\begin{equation}\label{eq:A}
A = -g^{-1}\dot{g}g^{-1}\bpartial_V g + g^{-1}\bpartial_V\dot{g},
\end{equation}
where $\dot{g}=(dg/dt)(0)$. A straightforward calculation shows that the conditions on $g$ imply that, with respect to the smooth
decomposition \eqref{eq:V4decomp}, it has the form
\[
g=\begin{pmatrix}
1&-av&-a^{-1}u&0\\0&a&0&u\\0&0&a^{-1}&v\\0&0&0&1
\end{pmatrix},
\]
where $a$ is a smooth $\Ct$-valued function on $\Sigma_c$ and $u\in\Gamma(K^{-1}L)$, $v\in\Gamma(K^{-1}L^{-1})$ satisfy
$\bpartial(uv)=0$. Since $K^{-2}$ has no globally holomorphic non-zero sections, either $u=0$ or $v=0$. We will treat the case 
$u=0$: the other case follows \emph{mutatis mutandis}. In this case $g$ has a $2\times 2$ block decomposition and we deduce that
\eqref{eq:A} holds if and only if $\eta_1=0$ and
\[
\begin{pmatrix} 0& -\eta_2\\ 0&\alpha \end{pmatrix}
= -g_1^{-1}\dot{g_1}g_1^{-1}\bpartial_{V_2} g_1 + g_1^{-1}\bpartial_{V_2}\dot{g_1},
\]
where
\[
\bpartial_{V_2} = \begin{pmatrix} \bpartial & -\beta_2\\ 0& \bpartial_L+\alpha \end{pmatrix}
\quad
g_1 = \begin{pmatrix} 1 &-av\\ 0& a\end{pmatrix}.
\]
However, this is precisely the condition that the deformation of the holomorphic structure of $V_2$ on $K^{-1}\oplus L$ along the
curve $[\bpartial_L+t\alpha,-\beta_2-t\eta_2]$ is constant. Hence $(\alpha,0,\eta_2)$ is tangent to the fibre of
$\hat\caW_{c,l}\to \caW_{c,l}$ at
the point $(0,\beta_1,\beta_2)$. We conclude that $d\psi$ has trivial kernel at $(L,[\beta])$.
\end{proof}

\subsection{Superminimal maps and Hodge bundles.} \label{superminimal-subsection}
An important geometric invariant of any equivariant minimal immersion $f$ is the holomorphic quartic differential
\begin{equation}\label{eq:U_4}
U_4 = \h{\II^{2,0}}{\II^{0,2}}=Q_W(\II^{2,0},\II^{2,0}).
\end{equation}
This vanishes at points where either $\II^{2,0}$ is zero or it is $Q_W$-isotropic.
\begin{rem}
It is easy to check that points where $\II^{2,0}$ is isotropic are points where $f$ has
\emph{circular ellipse of curvature} (i.e., the image of the unit
circle in each $T_z\caD$ under the map $T_z\caD\to T_z\caD^\perp$; $X\mapsto \II(X,X)$ is a circle).
\end{rem}

By a simple adaptation of the harmonic sequence arguments used in \cite{BolW}
for minimal immersions into $S^n$, it can be shown that
an equivariant minimal immersion $f:\caD\to \RH^4$ is determined up to congruence by the induced metric $\gamma$ and the
holomorphic quartic differential $U_4$. The following definition comes from the harmonic sequence theory.
\begin{defn}
We will say $f$ is \emph{superminimal} when $U_4$ vanishes identically.
\end{defn}
Let $v_\gamma$ denote the area form for the induced metric $\gamma=f^*g$. This is $\pi_1\Sigma$-invariant
and therefore lives on $\Sigma$. We will call its integral over $\Sigma$ the area of the equivariant minimal immersion.
From the Gauss equation \eqref{eq:Gauss} we deduce that it satisfies
\begin{equation}\label{eq:area}
\int_\Sigma v_\gamma = 4\pi(g-1) -\int_\Sigma \|\II^{2,0}\|^2_\gamma v_\gamma.
\end{equation}
For superminimal immersions we can relate the last term to the Euler number of the normal bundle $\chi(T\Sigma^\perp)$.
\begin{thm}\label{thm:supermin}
Suppose $f:\caD\to\RH^4$ is equivariant superminimal. Then one of the following holds: (i) $\chi(T\Sigma^\perp)> 0$ and
\begin{equation}\label{eq:supermin1}
\kappa^\perp = \|\II^{2,0}\|^2_\gamma =-(1+\kappa_\gamma),
\end{equation}
or, (ii) $\chi(T\Sigma^\perp)< 0$ and
\begin{equation}\label{eq:supermin2}
\kappa^\perp = -\|\II^{2,0}\|^2_\gamma =1+\kappa_\gamma,
\end{equation}
or, (iii) $\chi(T\Sigma^\perp)= 0$ and $f$ is totally geodesic. In all cases the area of $f$ is
\begin{equation}\label{eq:superminarea}
\int_\Sigma v_\gamma = 2\pi\big(2(g-1) - |\chi(T\Sigma^\perp)|\big).
\end{equation}
In particular, there are superminimal immersions for every value of $\chi(T\Sigma^\perp)$ with $|\chi(T\Sigma^\perp)|<2(g-1)$,
but these are only linearly full when $\chi(T\Sigma^\perp)\neq 0$.
\end{thm}
\begin{proof}
Since $\II^{2,0}=\theta_1+\theta_2$ and $L,L^{-1}$ are isotropic and paired by $Q_W$, we have
$U_4 = 2\theta_1\theta_2$.  So $U_4=0$ if and only if  either $\theta_1=0$ or $\theta_2=0$. Now $\theta_j=0$ if and
only if $\beta_j=0$. By the stability conditions
in Lemma \ref{lem:stab2}, if $\chi(T\Sigma^\perp)=\deg(L)> 0$ then $[\beta_2]\neq 0$  and therefore it must be
$\theta_1=0$. From Lemma
\ref{lem:supermin} in appendix \ref{sec:GCR} this implies that the normal curvature satisfies \eqref{eq:supermin1}. Similarly,
for $\deg(L)<0$ we have $\theta_2=0$, which yields \eqref{eq:supermin2}. From Theorem \ref{thm:RH4} we know that when
$l=\chi(T\Sigma^\perp)\neq 0$ we have families of linearly full immersions with either $[\beta_1]=0$ (provided $l>0$)
or $[\beta_2]=0$ (provided $l<0$).

Now if $\deg(L)=0$ then one of
\eqref{eq:supermin1} or \eqref{eq:supermin2} must still hold, but $\int_\Sigma\kappa^\perp v_\gamma=0$, hence $\kappa^\perp=0$
and therefore $\II^{2,0}=0$, i.e., $f$ is totally geodesic.
\end{proof}
By combining \eqref{eq:area} with Lemma \ref{lem:kappaperp} we obtain an area bound for every equivariant minimal
immersion, based on the connected component of the moduli space in which it lies.
\begin{cor}\label{cor:areabound}
The area of an equivariant minimal immersion $f:\caD\to\RH^4$ is bounded
above by the area of any superminimal immersion whose normal bundle has the same Euler number:
\begin{equation}\label{eq:allarea}
\int_\Sigma v_\gamma \leq 2\pi[2(g-1) - |\chi(T\Sigma^\perp)|].
\end{equation}
\end{cor}
The next result allows us to relate the structure of $\caM(\Sigma,\RH^4)$ to the topology of the moduli spaces of Higgs
bundles $\caH(\Sigma_c,SO_0(4,1))$.
\begin{prop}\label{prop:Hodge}
An equivariant minimal immersion $f:\caD\to\RH^4$ is superminimal if and only if its Higgs bundle is a Hodge bundle.
\end{prop}
\begin{proof}
We observed above that $U_4=0$ if and only if at least one of $\beta_1,\beta_2$ is identically zero.
According to Proposition 7.5 in \cite{ApaG} the Higgs bundle $(V,Q_V,\phi)$ is a Hodge bundle when $V$ decomposes into a direct
sum $V=\oplus_r (W_r\oplus W_{-r})$ of holomorphic subbundles for which
$W_r$ is the eigenbundle (with eigenvalue $ir$, $r\in\R$) of an infinitesimal gauge transformation $\psi\in\Gamma(\End(V))$
satisfying $\psi^t=-\psi$, $\nabla\psi=0$ and
\[
\left[\begin{pmatrix} \psi &0\\ 0& 0\end{pmatrix},\begin{pmatrix} 0 & \phi\\ \phi^t&0\end{pmatrix}\right]=
i\begin{pmatrix} 0 & \phi\\ \phi^t&0\end{pmatrix}.
\]
This last condition is equivalent to $\psi\phi=i\phi$. 
In particular: (i) $\im\phi\subset W_1$, (ii) $Q_V(W_a,W_b)=0$ unless $b=-a$,
in which case $Q_V$ pairs them dually. When $V$ has rank $4$ there are only two possibilities: $W_1$ has rank either one or
two. In the former case $W_1=\im\phi =K^{-1}$ and $V$ must be decomposable, with Higgs bundle
decomposition
\[
(V'\oplus 1,\Phi')\oplus (L\oplus L^{-1},0),
\]
where $V'=K^{-1}\oplus K$ and $\Phi'$ is just $\Phi$ restricted to $V'$. In this case polystability requires $\deg(L)=0$
and the corresponding
minimal immersion is totally geodesic into a copy of $\RH^2$. When $W_1$ has rank two it is either $V_1$ or
$V_2$, since it is $Q_V$-isotropic and contains $\im\phi$. The holomorphic splitting $V=W_1\oplus W_{-1}$
then implies that either $\beta_1=0$ or $\beta_2=0$.

Conversely, suppose $\beta_1=0$, then we have a holomorphic splitting $V=V_2\oplus V_{-2}$ where
$V_{-2}$ is the subbundle $L^{-1}\oplus K$ with the holomorphic structure induced from $\bar\partial_V$.
Define $\psi$ to have
$i$-eigenspace $V_2$ and $-i$-eigenspace $V_{-2}$. This ensures that $\psi$ is skew-symmetric for $Q_V$ and that
$\psi\phi=i\phi$. It is
also $\nabla$-parallel since it acts as the complex structure on $K^{-1}$ (which is K\"{a}hler)) and as $J$ on $L$
(which is parallel). Hence $(V,Q_V,\phi)$ is a Hodge bundle. The case of $\beta_2=0$ is argued similarly.
\end{proof}

\subsection{Some remarks on the structure of $\caM(\Sigma,\RH^4)$.}\label{sec:Morse}
Now that we have identified the Hodge bundles we can gain more insight into the structure of $\caM(\Sigma,\RH^4)$
and, in particular, how its
topology is related to that of each Higgs bundle moduli space $\caH(\Sigma_c,SO_0(4,1))$.  This is very similar to the structure
observed for $N=\CH^2$ in \cite[\S 6.3]{LofM15}. First let us note that for $G=SO_0(4,1))$ the topology of $\caH(\Sigma_c,G)$ itself
is nowhere near as well understood as the case $G=PU(2,1)$. It is not even clear how many connected components it has (see
\cite{ApaG}), although one does know that it is disconnected by the invariant $w_2(\rho)\in\Z_2$, which for us
equals $\chi(T\Sigma^\perp)\bmod 2$.

Consider $\caM(\Sigma,\RH^4)$ as a family over $\caT_g$ with fibres $\caM(\Sigma_c,\RH^4)$.
By Theorem \ref{thm:RH4}
\[
\caM(\Sigma_c,\RH^4)\simeq \bigcup_{|l|<2(g-1)} \caW_{c,l}.
\]
Let $\caS_{c,l}$ be the locus of superminimal immersions for conformal class $[c]$ and whose normal bundle has Euler
number $l$. Let $(L,[\xi_1,\xi_2])$ denote the $\Ct$-orbit of $(L,\xi_1,\xi_2)$ described in
\eqref{eq:Ct}). 
From the proof of Theorem \ref{thm:supermin} we see that
\[
\caS_{c,l} = \begin{cases} \{(L,[0,\xi_2])\in\caW_{c,l}\}\text{ for $l>1$},\\
\{(L,[\xi_1,0])\in\caW_{c,l}\}\text{ for $l<1$},\\
\{(L,[0,0])\}\text{ for $l=0$}.\\
\end{cases}
\]
When $l\neq 0$ we have $\caS_{c,l}\subset\caW_{c,l}$, but $S_{c,0}$ lies on the boundary of $\caW_{c,0}$.
Notice that when $l>1$ $\caS_{c,l}$ is isomorphic to the bundle $\P\caV_{c,l}$, when $l<1$ it
is isomorphic to $\P\iota^*\caV_{c,-l}$, and $S_{c,0}\simeq\Pic_0(\Sigma_c)$. 
For $l\neq 0$ we can view $\caW_{c,l}$ as a vector bundle over $\caS_{c,l}$. 
When $l>1$ there is a natural projection 
\[
\caW_{c,l}\to \caS_{c,l};\quad (L,[\xi_1,\xi_2])\mapsto (L,[0,\xi_2])
\]
whose fibre at $L$ is identifiable with $H^1(\Sigma_c,K^{-1}L)$, and a similar observation holds for $l<0$. We will now show 
how this structure is related to the
\emph{Hitchin function} $\fE(E,\Phi)=\tfrac12\|\Phi\|_{L^2}^2$ on $\caH(\Sigma_c,G)$. Note that with this normalisation $\fE$ 
gives the harmonic map energy of the harmonic map corresponding to $(E,\Phi)$, and therefore it is the area, in the sense
defined earlier, when this harmonic map is a minimal immersion.

By Lemma \ref{lem:tr=0} the image of $\psi:\caW_{c,l}\to \caH(\Sigma_c,G)$ in \eqref{eq:psi}
lies in the \emph{nilpotent cone}, the locus where $\tr\Phi^2=0$. Now we recall Hausel's theorem \cite{Hau}, which asserts that the 
nilpotent cone agrees with the downwards gradient flow of $\fE$.
Recall that $\fE$ is viewed as the moment map for the Hamiltonian action  of $S^1$ on $\caH(\Sigma_c,G)$ given by
$(E,\Phi)\mapsto (E,e^{i\theta}\cdot\Phi)$. The fixed points of this action (and hence the critical points of $\fE$)
are precisely the Hodge bundles. Moreover, by a theorem of Kirwan \cite[Thm 6.16]{Kir} the unstable manifold of the 
downward gradient flow from a critical manifold $C$ of $\fE$ agrees with 
\[
\{(E,\Phi): \lim_{\lambda\to\infty}(E,\lambda\Phi) \in C\}.
\]
Given these facts, we can prove the following.
\begin{prop}
For $l\neq 0$, $\psi(\caW_{c,l})$ lies in the unstable manifold of $\psi(\caS_{c,l})$ for the downward gradient flow of
$\fE$.
\end{prop}
Note that this result is reflected in the bound on area in Corollary \ref{cor:areabound}.
\begin{proof}
By the remarks above it suffices to show that if $(E,\Phi)$ is the image under $\psi$ of $(W,Q_W,[\xi])\in\caW_{c,l}$ then
$\lim_{\lambda\to\infty}(E,\lambda\Phi)$ lies in $\psi(\caS_{c,l})$. In fact we will show that
\begin{equation}\label{eq:lambda}
\psi(W,Q_W,[\lambda^{-1}\xi]) = (E,\lambda\Phi).
\end{equation}
To see this, let represent the holomorphic structure $\bar\partial_E$ on $E$ and the Higgs field $\Phi$, with respect to the smooth
isomorphism $E\simeq K^{-1}\oplus W\oplus K\oplus 1$, by
\[
\bar\partial_E = \begin{pmatrix} \bar\partial &-\beta^t & 0 & 0\\
0&\bar\partial_W&\beta&0\\
0&0&\bar\partial&0\\
0&0&0&\bar\partial
\end{pmatrix},
\qquad
\Phi=\begin{pmatrix} 0&0&0&1\\ 0&0&0&0\\ 0&0&0&0\\ 0&0&1&0 \end{pmatrix},
\]
where $\beta\in \caE^{0,1}(\Hom(K,W))$ is such that $-\beta^t$ represents $[\xi]$.
Now a simple calculation shows that for any $\lambda\in\Ct$ the constant gauge transformation
\[
g_{\lambda} = \begin{pmatrix} \lambda &0&0&0\\ 0&I_W&0&0\\ 0&0&\lambda^{-1}\\ 0&0&0&1\end{pmatrix}
\]
has the property that $g_\lambda\Phi g_\lambda^{-1} = \lambda\Phi$, while $g_\lambda^{-1}\bar\partial_E g_\lambda$ is obtained by
replacing $\beta$ with $\lambda^{-1}\beta$ above. In other words the Higgs bundle $\psi(W,Q_W,[\lambda^{-1}\xi])$ is gauge equivalent
to $(E,\lambda\Phi)$. Since $\psi$ is continuous,
\[
\lim_{\lambda\to\infty} (E,\lambda\Phi) = \psi(W,Q_W,\lim_{\lambda\to\infty}[\lambda^{-1}\xi]).
\]
Since 
\[
[\lambda^{-1}\xi_1,\lambda^{-1}\xi_2]=[\lambda^{-2}\xi_1,\xi_2] = [\xi_1,\lambda^{-2}\xi_2],
\]
we deduce that, for $l\neq 0$, $\psi(\caW_{c,l})$ lies in the unstable manifold the downward gradient flow of of $\psi(\caS_{c,l})$.
\end{proof}
The identity \eqref{eq:lambda} also helps us understand the boundary of each component $\caM_l(\Sigma,\RH^4)$ for $l\neq 0$.
Aparicio \& Garc\'ia-Prada have shown in \cite[Thm 8.4]{ApaG} that the smooth minima of the Hitchin functional are absolute
minima, i.e., their Higgs fields are zero. Polystability obliges the corresponding representations take values in a maximal
compact subgroup of $G$. These do not correspond to minimal immersions, but rather to constant (harmonic) maps of $\caD$
into $\RH^4$. By \eqref{eq:lambda} these lie on the ``boundary at infinity'' of $\caW_{c,l}$ (i.e., as $\lambda\to 0$), 
as the limit
of downward gradient flow. In particular, each component $\caM_l(\Sigma,\RH^4)$ has this common boundary, but to pass through
this boundary requires collapsing the immersion down to a constant map.

The structure of $\caW_{c,0}$ is a little different. By \eqref{eq:lambda} this is the unstable manifold of the downward gradient 
flow from $\caS_{c,0}$, which lies on the boundary of $\caM(\Sigma,\RH^4)$. By Theorem \ref{thm:supermin} $\cup_c
\caS_{c,0}$ contains all the totally geodesic immersions into a copy of $\RH^2$. For each of these the representation
$\rho$ is reducible: one factor provides a representation into $SO_0(2,1)$ in the conjugacy class $[c]$, and
the other factor is a representation into $SO(2)$ acting on the flat normal bundle of the copy of $\RH^2$ (this is the data carried
by $L\in\Pic_0(\Sigma_c)$, since $\Pic_0(\Sigma_c)$ is isomorphic to the moduli space of flat $S^1$-bundles over
$\Sigma$). But by Lemma \ref{lem:stab2} $\caW_{c,0}$ also has on its boundary the set
\[
\{(1,\xi_1,\xi_2):\C.\xi_1=\C.\xi_2,\ \xi_j\neq 0\}/\Ct\simeq (H^1(\Sigma_c,K^{-1})\setminus\{0\})/\Z_2.
\]
This part of the boundary adjoins $S_{c,0}$ at one point, the limit as $\xi_1,\xi_2\to 0$. 
The isomorphism arises from the fact that in each orbit there is, up to sign, a unique $a\in\Ct$ such that
$a\xi_1=a^{-1}\xi_2$: the orbit is mapped to $\pm a\xi_1$.
From the proof of Lemma \ref{lem:stab2} this part of the boundary corresponds to a single copy of $\caM(\Sigma_c,\RH^3)/\Z_2$. 
In terms of the minimal
surface geometry, the quotient by $\Z_2$ appears because if $f:\caD\to\RH^3\subset\RH^4$ then its 
orientation in $\RH^3$ can be reversed by an orientation preserving isometry of $\RH^4$ (rotation through $\pi$ in the normal
bundle of $f$ in $\RH^4$).  
Finally, $\caW_{c,0}$ has the same ``boundary at infinity'' as $\caW_{c,l}$.
In summary, the boundary of the connected component $\caM_0(\Sigma,\RH^4)$ (i.e., the minimal immersions with flat
normal bundle) in $\caT_g\times \caR(\pi_1\Sigma,G)$ contains multiple copies of the moduli
space $\caM(\Sigma,\RH^2)$, one for each pair $(c,L)$ of marked conformal structure $c$ and degree zero holomorphic line bundle $L$
over $\Sigma_c$, and it contains one copy of the moduli $\caM(\Sigma,\RH^3)/\Z_2$ of linearly full unoriented minimal immersions 
into $\RH^3$.
\begin{rem}
This common ``boundary at infinity'' of Higgs bundles with zero Higgs field is, of course, identifiable with the moduli
space of flat $SO(4)$-bundles over $\Sigma$. As Aparicio and Garc\'ia-Prada note, it is this subspace which carries all
the information about the connected components of $\caH(\Sigma_c,G)$. However, when it comes to understanding
$\caM(\Sigma,\RH^4)$ the topology of the space of absolute minima plays no role, since as a limit of minimal immersions all 
limit points are the same, viz, constant maps. 
\end{rem}

\appendix

\section{The Gauss-Codazzi-Ricci equations.}\label{sec:GCR}
The zero curvature equations for the connexion $\nabla^E$ on $E$ yield the Gauss-Codazzi-Ricci equations for the minimal
immersion $f:\caD\to\RH^n$.  It is convenient to calculate these in a local orthonormal frame, adapted to the splitting
\eqref{eq:Vdecomp}. In a conformal
coordinate chart $(U,z)$ on $(\Sigma,\gamma)$, let $Z=\partial/\partial z$ and for any smooth section $\sigma$ of $E$
write $Z\sigma$ to
mean $\nabla^E_Z\sigma$ and so forth. Let $f_0$ denote the length $-1$ section of $E$ which corresponds to the map
$f$, so that $f_0$ generates the trivial line bundle in the summand $E=V\oplus 1$. Let $s=\|Z\|_\gamma$ in the induced metric
$\gamma$, so that locally $\gamma = 2s^2|dz|^2$. Then $\|\bar Z\|_\gamma=s$ so that
\[
f_1 = s^{-1} Zf_0,\quad f_{-1} = s^{-1}\bar Zf_0,
\]
locally smoothly frame $K^{-1}$ and $K$ respectively. Now choose an oriented orthonormal frame $\nu_1,\ldots,\nu_{n-2}$ for
$T\Sigma^\perp$: this provides a complex frame for $W$. Finally, let
$\eta_{jk}$ be the connexion $1$-forms for the connexion in the normal bundle $W$, i.e.,
\[
\eta_{jk}(\bar Z)=\h{\bar Z \nu_j}{\nu_k}
\]
Altogether $f_1,\nu_1,\ldots,\nu_{n-2},f_{-1},f_0$ provides a $U(n,1)$-frame for $E=V\oplus 1$.
In this local frame the holomorphic structure on $E$ is determined by the equations
\begin{eqnarray}
\bar Z f_1 & = & -(\bar Z\log s) f_1 + sf_0,\\
\bar Z \nu_j &=& \sum_k\eta_{jk}(\bar Z)\nu_k -s^{-1}\h{\II(\bar Z,\bar Z)}{\nu_j}f_1\\
\bar Zf_{-1} &=& (\bar Z\log s) f_{-1} + s^{-1} \II(\bar Z,\bar Z)\\
\bar Z f_0 &=& sf_{-1}.
\end{eqnarray}
Note that if we consider $dz$ as a local section of $K$ then $dz = s^{-1}f_{-1}$ and therefore by comparison with
\eqref{eq:beta}
\[
\beta(\bar Z):f_{-1}\mapsto s^{-1}\II(\bar Z,\bar Z).
\]
In this local frame the zero curvature equations for $\nabla$ are
\begin{eqnarray}\label{eq:GCR}
-s^{-2}Z\bar Z\log s^2  +s^{-4}\|\II(Z,Z)\|^2 +1& = &0,\\
\bar Z\h{\II(Z,Z)}{\nu_k}+\sum_j\eta_{jk}(\bar Z)\h{\II(Z,Z)}{\nu_j}&=&0,\\
R^\perp(Z,\bar Z) + s^{-2}[\II(Z,Z)\otimes\II(Z,Z)^* - \II(\bar Z,\bar Z)\otimes\II(\bar Z,\bar Z)^*] &=&0,
\end{eqnarray}
where $R^\perp(X,Y) = [\nabla^\perp_X,\nabla^\perp_Y]-\nabla^\perp_{[X,Y]}$ is the curvature in the normal bundle,
$\eta$ is the $\End(W)$ valued connexion $1$-form for the normal bundle connexion,
\[
\eta:\nu_j\mapsto \h{d\nu_j}{\nu_k}\nu_k,
\]
and
\[
\II(Z,Z)\otimes\II(Z,Z)^*:W\to W;\quad  \sigma\mapsto \II(Z,Z)\h{\sigma}{\II(Z,Z)}.
\]
Note that $R^\perp = d\eta+\eta\wedge \eta$.
The local expression for the Gaussian curvature for the induced metric $\gamma=2s^2|dz|^2$ is
\[
\kappa_\gamma = -s^{-2}Z\bar Z\log s^2,
\]
so that the Gauss equation in global form is
\begin{equation}\label{eq:Gauss}
\kappa_\gamma = -1-\|\II^{2,0}\|_\gamma^2.
\end{equation}
The Codazzi equation simply says $\nabla^\perp_{\bar
Z}\II(Z,Z)=0$, i.e., $\II^{2,0}$ is a holomorphic quadratic differential with values in the complexified
normal bundle $W$.
It follows that $\kappa_\gamma=-1$ either everywhere (for totally geodesic embeddings) or only at isolated points.

In general the Ricci equation does not simplify further. But for maps into $\RH^4$ the normal bundle curvature can be
represented by a scalar $\kappa^\perp$. This is defined by
\begin{equation}\label{eq:kappaperp}
\h{R^\perp\nu_1}{\nu_2} = \kappa^\perp v_\gamma,
\end{equation}
where $v_\gamma$ is the area form with respect to the induced metric $\gamma$.
The left hand is side is well-defined globally because the normal bundle connexion is $SO(2)$-invariant.
This definition ensures that the normal bundle curvature is related to the Euler number of the normal bundle by
\begin{equation}\label{eq:chiperp}
\chi(T\Sigma^\perp) = \frac{1}{2\pi} \int_\Sigma \kappa^\perp v_\gamma.
\end{equation}
\begin{lem}\label{lem:kappaperp}
For a minimal immersion $f$ into $\RH^4$
\begin{equation}\label{eq:globkappaperp}
(\kappa^\perp)^2 = \|\II^{2,0}\|^4_\gamma-\|U_4\|^2_\gamma = (1+\kappa_\gamma)^2 -\|U_4\|^2_\gamma,
\end{equation}
where $U_4=\h{\II^{2,0}}{\II^{0,2}}$.
\end{lem}
\begin{proof}
Set $A_j = \h{\II(Z,Z)}{\nu_j}$, and write $U_4=u_4dz^4$, so that
\[
\|\II(Z,Z)\|^2 = |A_1|^2+|A_2|^2, \quad u_4 = A_1^2+A_2^2,
\]
Then the local form of the Ricci equation is
\begin{eqnarray}
i\kappa^\perp s^2 -s^{-2}(A_1\bar A_2 -\bar A_1 A_2) &= &0.\label{eq:ricci}
\end{eqnarray}
Define $s_2 = \sqrt{|A_1|^2+|A_2|^2}$ so that $s^{-2}s_2$ is the local expression for $\|\II^{2,0}\|_\gamma$. Then
\begin{eqnarray*}
s_2^4& = &  |A_1|^4 + 2|A_1|^2|A_2|^2 + |A_2|^4,\\
|u_4|^2 & = & |A_1|^4 + A_1^2 \bar A_2^2 + \bar A_1^2 A_2^2 + |A_2|^4.
\end{eqnarray*}
Therefore, by \eqref{eq:ricci},
\begin{equation}\label{eq:kappaperp2}
(\kappa^\perp)^2= s^{-8}(s_2^4-|u_4|^2).
\end{equation}
This gives the global equation \eqref{eq:globkappaperp}.
\end{proof}
\begin{lem}\label{lem:supermin}
Suppose $f$ is superminimal, i.e., $U_4=0$. Then either $\kappa^\perp = \|\II^{2,0}\|^2_\gamma$ or
$\kappa^\perp = -\|\II^{2,0}\|^2_\gamma$.
\end{lem}
\begin{proof}
Since $u_4=A_1^2+A_2^2 = (A_1+iA_2)(A_1-iA_2)$ this vanishes if and only if either $A_1=iA_2$ or $A_1=-iA_2$.
By \eqref{eq:thetaj}, if $\theta_1=0$ then $A_1=iA_2$ and therefore \eqref{eq:ricci} becomes
\[
i\kappa^\perp = is^{-4}(|A_1|^2+|A_2|^2),
\]
and therefore $\kappa^\perp = \|\II^{2,0}\|^2_\gamma$. Similarly, when $\theta_2=0$ the opposite equality is obtained.
\end{proof}

\section{Proof of Lemma \ref{lem:maxdeg}.} \label{appendix-b-section}

We follow the terminology of \cite{LanN}. For any holomorphic vector bundle $\xi$ of rank two over $\Sigma_c$ define
\begin{eqnarray*}
s(\xi) & =  &c_1(\xi) - 2\max\{\deg(\lambda):\lambda\subset \xi\text{ a holomorphic line subbundle}\}\\
&=& c_1(\xi) - 2\mu(\xi).
\end{eqnarray*}
It is known that $s$ is lower semi-continuous on algebraic families of vector bundles. We are interested in non-trivial
extension bundles of the form $0\to K^{-1}\to\xi\to L\to 0$ where $1\leq l<2(g-1)$ for $l=\deg(L)$. We want to show that
the set $\caV_l^0$ of all such extensions with $\mu(\xi)<0$, equally $s(\xi)>l-2(g-1)$, is non-empty and open.
By lower semi-continuity, it suffices to show that this is non-empty.

Since $s(K\otimes\xi) = s(\xi)$ it is equivalent to consider extension bundles of the form
\[
0\to 1\to \xi\to \lambda\to 0,
\]
for which  $d=\deg(\lambda)= l+2(g-1)$ and $s(\xi)>l-2(g-1)$. By Serre duality $H^1(\lambda^{-1})\simeq H^0(K\lambda)^*$ and
by Riemann-Roch this space has dimension
\[
h^0(K\lambda)= l+3(g-1).
\]
Consider the embedding
\[
\varphi_\lambda:\Sigma_c\to \P H^0(K\lambda)^*,
\]
which assigns to each $p\in\Sigma_c$ the hyperplane $H^0(K\lambda(-p))$ of sections of $K\lambda$ which vanish at $p$ (this is
well-defined since the degree of $K\lambda$ is sufficiently high for it to be very ample). Relative to this the $l$-th secant
variety $\Sec_l(\Sigma_c)$ is the subvariety of $\P H^0(K\lambda)^*$ whose elements correspond to linear forms which vanish
on $H^0(K\lambda(-D))$ for some effective divisor $D$ of degree $l$. The following result is Prop.\ 1.1
of \cite{LanN} for the case where $d= l+2(g-1)$ and using the value $s=l+2-2(g-1)$.
In particular, the conditions $s\equiv d\bmod 2$ and $4-d\leq s\leq d$ of that proposition hold when $l\geq 1$.
\begin{lem}
The bundle $\xi$ has $s(\xi)\geq l+2-2(g-1)$ if and only if $\xi\not\in \Sec_l(\Sigma_c)$.
\end{lem}
Now
\[
\dim(\Sec_l(\Sigma_c))=2l-1< l+3g-4=\dim(\P H^0(K\lambda)^*),
\]
and therefore $\Sec_l(\Sigma_c)$ is a proper subvariety. It follows that there exist non-zero $\xi\in H^1(\lambda^{-1})$
with $s(\xi)\geq l+2-2(g-1)>l-2(g-1)$. We deduce that $\caV_l^0\neq \emptyset$.

\end{document}